\def\N{\mathbb N}
\def\Z{\mathbb Z}
\def\R{\mathbb R}
\def\Q{\mathbb Q}
\def\C{\mathbb C}

\def\O{\mathcal{O}}
\def\Mm{\mathcal{M}}
\def\pfz{\begin{proof}}
\def\pfk{\end{proof}}
\documentclass[12pt]{amsart}
\usepackage[lmargin=2.5cm, rmargin=2.5cm,bottom=2.5cm,top=2.5cm]{geometry}
\usepackage{amssymb,amsmath,mathdots,amsthm,enumerate}
\usepackage{graphicx}

\usepackage{epsfig}
\usepackage{tikz}
	\usetikzlibrary{tikzmark, calc}
	\usetikzlibrary{automata,positioning}
    \usetikzlibrary{shapes.misc}
	\tikzset{every state/.style={inner sep=0pt,minimum size=28pt}}

\usepackage{subfigure}
\usepackage{comment}
\usepackage{algorithm,algorithmic}
\usepackage{enumerate}

\usepackage{color}

\definecolor{darkgreen}{rgb}{0,0.5,0} 
\usepackage[
        colorlinks, citecolor=darkgreen,
        backref,
        pdfauthor={Zuzana Masáková, Tomáš Vávra, Francesco Veneziano}, 
        pdftitle={Finiteness of continued fractions over quadratic number fields},
]{hyperref}
\usepackage[alphabetic,backrefs]{amsrefs} 

\providecommand{\abs}[1]{\left\lvert{#1}\right\rvert}

\newcommand{\floor}[1]{\left\lfloor{#1}\right\rfloor}
\renewcommand{\O}{\mathcal{O}}
\newcommand{\Qbar}{\overline{\mathbb{Q}}}
\newcommand{\CFF}{{\textup{(CFF)}}}
\newcommand{\CFP}{{\textup{(CFP)}}}

\DeclareMathOperator{\Gal}{Gal}

\newtheorem{thm}{Theorem}[section]
\newtheorem{prop}[thm]{Proposition}
\newtheorem{coro}[thm]{Corollary}
\newtheorem{lem}[thm]{Lemma}
\newtheorem{de}[thm]{Definition}

\newtheorem{remark}[thm]{Remark}
\newtheorem{conj}[thm]{Conjecture}

\newtheorem*{problem*}{Problem}

\title[Continued fractions over quadratic number fields]{Finiteness and periodicity of continued fractions over quadratic number fields}

\author{Zuzana Mas\'akov\'a}
\address{Department of Mathematics, FNSPE, Czech Technical University in Prague, Trojanova 13, 120~00 Praha 2, Czech Republic}
\email{zuzana.masakova@fjfi.cvut.cz}

\author{Tom\'a\v s V\'avra}
\address{Department of Algebra, FMP, Charles University, Sokolovsk\'a 83, 186~75 Praha 8, Czech Republic}
\email{tomas.vavra@mff.cuni.cz}

\author{Francesco Veneziano}
\address{Department of mathematics, University of Genova, Via Dodecaneso 35, 16146 Ge\-no\-va, Italy}
\email{veneziano@dima.unige.it}

\date{\today}

\begin{document}

\begin{abstract}
We consider continued fractions with partial quotients in the ring of integers of a quadratic number field $K$ and we prove a generalization to such continued fractions of the classical theorem of Lagrange.
A particular example of these continued fractions is the $\beta$-continued fraction introduced by Bernat. As a corollary of our theorem we show that for any quadratic Perron number $\beta$, the $\beta$-continued fraction expansion of elements in $\Q(\beta)$ is either finite of eventually periodic. The same holds for $\beta$ being a square root of an integer.
We also show that for certain 4 quadratic Perron numbers $\beta$, the $\beta$-continued fraction represents finitely all elements of the quadratic field $\Q(\beta)$, thus answering questions of Rosen and Bernat. Based on the validity of a conjecture of Mercat, these are all quadratic Perron numbers with this feature.
\end{abstract}

\maketitle
\allowdisplaybreaks

\section{Introduction}

In 1977, Rosen~\cite{rosen-amm} stated the following research problem: ``Is it possible to devise a continued fraction that represents uniquely all real numbers, so that the finite continued fractions represent the elements of an algebraic number field, and conversely, every element of the number field is represented by a finite continued fraction?" The classical regular continued fraction has this property with respect to the field of rational numbers.

For $\lambda=2\cos\frac{\pi}{q}$ with $q\geq 3$ odd, Rosen gives a definition of $\lambda$-continued fractions, whose partial quotients  are integral multiples of $\lambda$.
Rosen shows, as a consequence of his own work \cite{rosen}, that if $q=5$ (i.e.\ $\lambda=\varphi=\frac12(1+\sqrt{5})$ the golden ratio) the $\lambda$-continued fraction satisfies his desired property.

A rather different construction was presented by Bernat~\cite{bernat}. Here he defines $\varphi$-continued fractions whose partial quotients belong to the set of the so-called $\varphi$-integers, i.e.\ numbers whose greedy expansion in base $\varphi$ uses only non-negative powers of the base. Bernat shows that his $\varphi$-continued fractions also represent every element of $\Q(\sqrt{5})$ finitely.
His proof is established using a very detailed and tedious analysis of the behaviour of $\varphi$-integers under arithmetic operations. This approach depends crucially on the arithmetic properties of $\varphi$-integers, descending from the fact that $\varphi$ is a quadratic Pisot number.

It is stated in~\cite{bernat} as an open question, whether the analogously-defined continued fraction expansion based on the $\beta$-integers would provide finite representation of $\Q(\beta)$ for any other choice of a quadratic Pisot number $\beta$. When trying to adapt Bernat's proof to other values of $\beta$, already in the case of the next smallest quadratic Pisot number $\beta=1+\sqrt{2}$ the necessary analysis becomes even more technical, preventing one from proving the finiteness of the expansions.
So far, it was not even known whether $\beta$-continued fractions provide at least an eventually periodic representation of all elements of $\Q(\beta)$.

\medskip

In this paper we have taken a different approach, considering more general continued fractions whose partial quotients belong to some discrete subset $M$ of the ring of integers in a real quadratic field $K$. The $\beta$-continued fraction of Bernat is a special case of such $M$-continued fractions when $M$ is chosen to be the set $\Z_\beta$ of $\beta$-integers (see Section~\ref{sec:beta-integers}).
With the aim of answering Bernat's question and classifying all quadratic numbers $\beta$ according to the qualitative behaviour of the $\beta$-continued fraction expansion of elements of $\Q(\beta)$, we introduce the following definitions.

\medskip

Let $\beta>1$ be a real algebraic integer.
\begin{itemize}
  \item[$\CFF$] We say that $\beta$ has the Continued Fraction Finiteness property $\CFF$ if every element of $\Q(\beta)$ has a finite $\beta$-continued fraction expansion.
  \item[$\CFP$] We say that $\beta$ has the Continued Fraction Periodicity property $\CFP$ if every element of $\Q(\beta)$ has a finite or eventually periodic $\beta$-continued fraction expansion.
\end{itemize}

We show that all quadratic Perron numbers and the square roots of positive integers satisfy $\CFP$ (Theorems~\ref{t:quadraticperiodic} and \ref{t:squareroots}). We prove $\CFF$ for four quadratic Perron numbers including the golden ratio $\varphi$ (Theorem~\ref{t:3perron}). Moreover, assuming a conjecture stated by Mercat~\cite{mercat}, we show that these four Perron numbers are the only ones with $\CFF$.

 We also consider the case of non-Perron quadratic $\beta$. In this case, if the algebraic conjugate of $\beta$ is positive we are able to construct a class of elements in $\Q(\beta)$ having aperiodic $\beta$-continued fraction expansion, thus showing that neither $\CFF$ nor $\CFP$ hold (Theorem~\ref{t:nonperronklad}).
To achieve a full characterisation of the numbers for which $\CFF$ nor $\CFP$ hold, it remains to give complete answer when $\beta$ is a non-Perron number with a negative algebraic conjugate.
Computer experiments suggest that $\CFF$ is not satisfied by any such $\beta$, with some of them having $\CFP$ and some not.
In Theorem~\ref{t:negatconjugnotCFF} we are able to construct counterexamples to $\CFF$ for a large class of $\beta$, but even assuming Mercat's conjecture the matter is not fully settled.

\medskip

These theorems on $\beta$-continued fractions are obtained with a wide rage of different techniques ranging from diophantine approximation to algebraic number theory and dynamics.

As our main result, we consider a convergent continued fraction $[a_0,a_1,a_2,\dots]$ with value in a quadratic number field $K$ and  partial quotients in the ring of integers of $K$; denote by $a'_i$ the image of $a_i$ under the (nontrivial) Galois automorphism of $K/\Q$. If each partial quotient satisfies $|a'_i|\leq a_i$ we show that the continued fraction is eventually periodic.
If we assume that all $a_i$ belong to $\Z$, then we recover the statement of the classical Lagrange's theorem.

\medskip

One could also study the properties $\CFF$, $\CFP$ for algebraic integers $\beta$ of degree bigger than 2. The issue is considerably more intricate, and already in the cubic Pisot case we can find instances of different $\beta$'s that seem to have $\CFF$, $\CFP$, and aperiodic $\beta$-continued fraction expansions. A modification of some of our arguments may be possible in the case of higher degree fields, but its application is likely to be highly nontrivial.

Another setting in which the same problem might be stated is that of function fields: a function field analogue of $\beta$-continued fractions has been studied in~\cite{HbKa}.

We also mention, as a different path of inquiry that could lead to the study of $M$-continued fractions for other choices of $M$, the recent work \cite{elkies}. Here the authors studied periodic continued fractions with fixed lengths of preperiod and period and with partial quotients in the ring of $S$-integers of a number field, describing such continued fractions as $S$-integral points on some suitable affine varieties.

\section{Preliminaries}
\subsection{Continued fractions}\label{subsec:CFintro}
Let $(a_i)_{i\geq 0}$ be a sequence of real numbers such that $a_i>0$ for $i\geq 1$, and define two sequences $(p_n)_{n\geq -2}$, $(q_n)_{n\geq -2}$ by the linear second-order recurrences
\begin{equation}\label{eq:pnqn}
\begin{aligned}
p_{n}&=a_np_{n-1}+p_{n-2}, &p_{-1}=1, p_{-2}=0,\\
q_{n}&=a_nq_{n-1}+q_{n-2}, &q_{-1}=0, q_{-2}=1.
\end{aligned}
\end{equation}
This recurrence can be written in matrix form as
$$
\begin{pmatrix}
  p_{n+1} & p_{n} \\
  q_{n+1} & q_n
\end{pmatrix} =
\begin{pmatrix}
  p_{n} & p_{n-1} \\
  q_{n} & q_{n-1}
\end{pmatrix}
\begin{pmatrix}
  a_{n+1} & 1 \\
  1 & 0
\end{pmatrix}, \quad n\geq -1.
$$
Taking determinants, it can be easily shown that
\begin{equation}\label{eq:nesoudelnost}
  p_{n-1}q_n-p_nq_{n-1} = (-1)^n, \quad n\geq -1.
\end{equation}

By induction, one can also show that
$$
[a_0,a_1,\dotsc,a_n]:=a_0+\frac{1}{a_1+\cfrac{1}{\ddots+\cfrac{1}{a_{n-1}+\cfrac{1}{a_n}}}}=\frac{p_n}{q_n}.
$$
Notice that the assumption of positivity of $a_i$ ensures that the expression on the right-hand side is well-defined for every $n$.

\subsubsection{Continuants}The numerator and denominator $p_n$, $q_{n}$ can be expressed in terms of the $a_i$ using the so-called continuants: multivariate polynomials defined by the recurrence
\begin{align*}
K_{-1}=0,\ K_0=1,\ K_n(t_1,\dots,t_n)=t_nK_{n-1}(t_1,\dots,t_{n-1})+K_{n-2}(t_1,\dots,t_{n-2}).
\end{align*}
Then $p_n=K_{n+1}(a_0,\dots,a_n)$ and $q_{n}=K_{n}(a_1,\dots,a_{n})$.
It is clear from the definition that $K_n$ is a polynomial with positive integer coefficients in $t_1,\dotsc,t_{n}$, and that each of the $t_i$ appears in at least one monomial with a non-zero coefficient.

The continuants satisfy a number of useful properties. We will in particular need that
$K_{n}(t_1,\dots,t_n)=K_n(t_n,\dots,t_1)$ and that for all $k,l\geq 1$ we have
\begin{multline}\label{eq:continuant}
K_{k+l}(t_1,\dots,t_{k+l})=
      K_{k}(t_1,\dots,t_k)K_{l}(t_{k+1},\dots,t_{k+l})+ \\ + K_{k-1}(t_1,\dots,t_{k-1})K_{l-1}(t_{k+2},\dots,t_{k+l}).
\end{multline}
We refer the reader to \cite{HarWri}*{Chapter X} for the classical theory of continued fractions and the proofs of the properties which we list in this section. Properties of continuants are given in~\cite{ConcreteMaths}*{Section 6.7}.

\subsubsection{Convergence}For an infinite sequence of $a_i$, $i\geq 1$, the continued fraction $[a_0,a_1,a_2,\dots]$ is defined as the limit
\begin{equation}\label{eq:defcf}
[a_0,a_1,a_2,\dots]:=\lim_{n\to\infty} \frac{p_n}{q_n}\,,
\end{equation}
if the limit exists. The numbers $a_i$ are called partial quotients, and the fractions $\frac{p_n}{q_n}$ the convergents of the continued fraction $[a_0,a_1,a_2,\dots]$.

Under our assumptions (positivity of the $a_i$), the limit~\eqref{eq:defcf} exists if and only if $\lim_{n\to\infty}q_n=+\infty$, and a sufficient condition for it is that $\inf_{i} a_i > 0$. (If $c=\inf_{i} a_i$ then $q_n\gg \left(1+\frac{c}{2}\right)^n$.)

\subsubsection{Approximation properties}Suppose that the continued fraction converges.
For each $i\geq -1$, we have that
\begin{equation}\label{eq:raclom}
\xi=a_0+\cfrac{1}{a_1+\cfrac{1}{\dots\
+\cfrac{1}{a_i+\cfrac1{\xi_{i+1}}}}} = \frac{\xi_{i+1}p_i+p_{i-1}}{\xi_{i+1}q_i+q_{i-1}}\,,
\end{equation}
where the $\xi_i=[a_i,a_{i+1},\dotsc]$ are the so-called complete quotients.

From~\eqref{eq:raclom}, we derive
\begin{equation}\label{eq:rozdil}
\xi - \frac{p_i}{q_i} =
\frac{p_{i-1}q_i-q_{i-1}p_i}{q_i(\xi_{i+1}q_i+q_{i-1})}
= \frac{(-1)^i}{q_i(\xi_{i+1}q_i+q_{i-1})}\,,
\end{equation}
where we used~\eqref{eq:nesoudelnost}.

As $(q_j)_{j\geq 1}$ tends to infinity and the convergents $p_i/q_i$ tend to $\xi$, equation~\eqref{eq:rozdil} allows us to quantify the rate of convergence with the following estimates:
\begin{equation}\label{eq:odhad}
  \abs{\xi - \frac{p_i}{q_i}}
= \frac{1}{q_i(\xi_{i+1}q_i+q_{i-1})}\leq \frac{1}{q_i(a_{i+1}q_i+q_{i-1})}=\frac{1}{q_{i}q_{i+1}}\leq\frac{1}{a_{i+1} q_{i}^2}\,.
\end{equation}
If $\inf_i a_i\geq 1$, then $\xi_{i+1}\leq a_{i+1}+1$, and we also have the lower estimate
\begin{equation*}
 \abs{\xi-\frac{p_i}{q_i}}\geq\frac{1}{q_i q_{i+2}}.
\end{equation*}

\subsubsection{Uniqueness of the expansion}When the partial quotients of a continued fraction $[a_0,a_1,\dots]$ take values in the integers and $a_i\geq 1$ for $i\geq 1$, then the so-called regular continued fractions
represent uniquely any real number $\xi=\lim_{n\to\infty} \frac{p_n}{q_n}$, except for the ambiguity $[a_0,\dots,a_n]=[a_0,\dots,a_n-1,1]$ in the finite continued fractions representing rational numbers.

Something similar can be shown in a slightly more general setup.
\begin{lem}\label{l:unique}
 Let $m>0$ and $A\subseteq \R$ such that
 \begin{itemize}
  \item $a\geq m\quad \forall a\in A$;
  \item $\abs{a-b}\geq 1/m\quad \forall a,b\in A$ distinct.
 \end{itemize}
Then every real number has at most one expression as an infinite continued fraction with partial quotients in $A$.
\end{lem}

\begin{proof}
 Let $\xi=[a_0,a_1,a_2,\dotsc]$ be a convergent continued fraction with partial quotients in $A$. We have that
 \[a_0<\xi=a_0+\frac{1}{\xi_1}<a_0+\frac{1}{a_1}<a_0+\frac{1}{m},\]
which implies
 \[
  \xi-\frac{1}{m}<a_0<\xi.
 \]
 Thanks to the hypothesis that no two elements of $A$ are less than $1/m$ apart, this identifies uniquely the value of $a_0$, and thus also that of $\xi_1$. The statement now follows by induction.
\end{proof}
\subsubsection{Finiteness and periodicity}It is a well known fact that finite regular continued fractions represent precisely the rational numbers. Lagrange's theorem states that irrational quadratic numbers have an eventually periodic regular continued fraction expansion. Galois proved that a quadratic number $\xi$ has a purely periodic regular continued fraction if and only if $\xi>1$ and its algebraic conjugate $\xi'$ belongs to the interval $(-1,0)$.

\subsection{The Weil height of algebraic numbers}
The Weil height is an important tool that measures, broadly speaking, the arithmetic complexity of algebraic numbers.
It can be described in different ways, but the usual definition involves an infinite product decomposition in local factors.

Let $K$ be a number field, and $\Mm_K$ a set of representatives for the places of $K$ (i.e. equivalence classes of non-trivial absolute values over $K$), suitably normalized in such a way that a product formula $\prod_{v\in\Mm_K}\abs{x}_v=1$ holds for all $x\in K^*$.
Then we can define the (multiplicative) Weil height as
\[
 H(x)=\prod_{v\in\Mm_K}\sup(1,\abs{x}_v),
\]
where it is easily seen that all but finitely many factors of the infinite product are equal to 1.
Thanks to the choice of the normalization, this definition extends to a function $H:\Qbar\to [1,+\infty)$ where $\Qbar$ denotes the algebraic closure of rational numbers. The function $H$ has the following properties:

\begin{prop}\label{prop:heights}
For all non-zero $x,y\in\Qbar$ we have
 \begin{enumerate}[(i)]
 \item $H(x+y)\leq 2 H(x)H(y)$;
 \item $H(xy)\leq H(x)H(y)$;
 \item $H(x^n)=H(x)^{\abs{n}}$ for all $n\in\Z$;
 \item $H(\sigma(x))=H(x)$ for all $\sigma\in\Gal(\Qbar/\Q)$;
 \item $H(x)=1$ if and only if $x$ is a root of unity (Kronecker's theorem);
 \item\label{northcott} for all $B,D>0$ the set $\{\xi\in\Qbar\mid H(\xi)\leq B \text{ and }[\Q(\xi):\Q]\leq D\}$ is finite (Northcott's theorem).
\end{enumerate}
\end{prop}
We refer the reader to the first two chapters of \cite{BG} for a thorough introduction to the theory of heights and the proof of the properties listed above.

\begin{remark}\label{rem:Mahlermeasure}
 Let $\xi\in \Qbar$, and let $d$ be the (positive) leading coefficient of the minimal polynomial of $\xi$ over $\Z$. Then
 \[
  H(\xi)^{[\Q(\xi):\Q]}=d \prod_{\sigma\in\Gal(\Q(\xi)/\Q)}\sup\left(1,\abs{\sigma(\xi)}\right).
 \]
\end{remark}

We will only use heights of numbers in real quadratic fields, so we give the exact normalization of the absolute values in this case.
Let $K$ be a real quadratic number field, and let $\mathfrak{p}$ be a prime ideal of $\O_K$ lying over a prime number $p$. Then we normalize $\abs{\cdot}_\mathfrak{p}$ in such a way that $\abs{p}_\mathfrak{p}=p^{-1}$ if the prime $p$ is inert or ramifies in $\O_K$, and $\abs{p}_\mathfrak{p}=p^{-1/2}$ if the prime $p$ splits. The two non-archimedean absolute values are normalized as follows: $\abs{x}_+=\abs{x}^{1/2}$ and $\abs{x}_-=\abs{x'}^{1/2}$, where by $x'$ we denote the algebraic conjugate of $x$ in $K$.

\section{Continued Fractions over Quadratic Fields.}\label{sec:MCF_in_quadratic}

In this section we focus on continued fractions with positive partial quotients in the ring of integers of a real quadratic field.
Let $K$ be a real quadratic field and let ${\mathcal O}_K$ be its ring of integers. For an element $x\in K$ we denote by $x'$ its image under the unique non-trivial automorphism of $K$. The main result of the paper is the following theorem.

\begin{thm}\label{t:superteorem}
Let $K$ be a real quadratic field. Let $\xi=[a_0,a_1,\dotsc]$ be an infinite continued fraction with $a_n\in{\mathcal{O}_K}$ such that $a_n\geq 1$ and $|a'_n|\leq a_n$ for $n\geq 0$. Assume that $\xi\in K$. Then the sequence $(a_n)_{n\geq 0}$ is eventually periodic with all partial quotients in the period belonging to $\Z$, or all belonging to $\sqrt{D}\Z$.

Moreover, there is an effective constant $C_\xi$ estimating the length of the preperiod.
\end{thm}
Notice that if all $a_i$ belong to $\Z$ we recover the statement of the classical Lagrange's theorem.

We start with a lemma controlling the growth of the complete quotients.

\begin{lem}\label{l:completequobound}
Let $K$ be a real quadratic field.
 Let $\xi=[a_0,a_1,\dotsc]$ be an infinite continued fraction with $a_n\in{\mathcal{O}_K}$ such that $a_n>0$ for $n\geq 0$. Assume that $\xi\in K$.
 Then the height of the complete quotients $\xi_{n}$ can be estimated as
 \begin{equation*}
     H(\xi_{n+1})\leq \frac{H(\xi)}{\inf(q_{n},q_{n+1})^{1/2}}\sup\left(|\xi' q'_n-p'_n|,|\xi' q'_{n-1}-p'_{n-1}|\right)^{1/2}.
 \end{equation*}
 If furthermore $a_n\geq 1$ and $\abs{a_n'}\leq a_n$ for all $n\geq 0$, then
 \begin{align*}
     H(\xi_{n})&\leq \sqrt{3}H(\xi),\\
     H(a_{n})&\leq a_n \leq 3 H(\xi)^2,
 \end{align*}
 and therefore only finitely many distinct complete quotients and partial quotients may occur.
\end{lem}

\begin{proof}
Denote for simplicity $A_n:=\xi q_n - p_n\in \O_{K}+\xi\O_{K}$. By~\eqref{eq:odhad} we have $\abs{A_n}\leq q_{n+1}^{-1}$. From the ultrametric inequality and the integrality of $p_n,q_n$ it follows that, for all non-archimedean absolute values of $K$,
\begin{equation}\label{eqn:nonarch.abs.val.An}
\abs{A_n}_v\leq\sup\left(\abs{\xi q_n}_v,\abs{p_n}_v\right)\leq \sup\left(\abs{\xi}_v,1\right).
\end{equation}
Due to~\eqref{eq:raclom}, the complete quotients $\xi_n$ can be expressed as
 \begin{equation*}
  \xi_{n+1}=-\frac{A_{n-1}}{A_n}\quad \forall n\geq -1.
 \end{equation*}
 Consider now, for all $n\geq 0$, the following chain of inequalities
 \begin{multline*}
  H(\xi_{n+1})=\prod_{v}\sup\left(\abs{A_{n-1}/A_{n}}_v,1\right)=\prod_{v}\sup\left({\abs{A_{n-1}}_v,\abs{A_{n}}_v}\right)=\\
  =\sup\left({\abs{A_{n-1}},\abs{A_{n}}}\right)^{1/2}\sup\left({\abs{A_{n-1}'},\abs{A_{n}'}}\right)^{1/2}\prod_{v\text{ finite}}\sup\left({\abs{A_{n-1}}_v,\abs{A_{n}}_v}\right)\leq\\
  \leq \inf(q_{n},q_{n+1})^{-1/2}\sup\left({\abs{A_n'},\abs{A_{n-1}'}}\right)^{1/2}H(\xi),
 \end{multline*}
where the second equality follows from the product formula and the last inequality from \eqref{eqn:nonarch.abs.val.An}. This proves the first part of the statement.

Under the condition that $a_n\geq 1$ we know that the sequences of the $p_n$ and $q_n$ are strictly increasing, and under the assumption that $\abs{a_n'}\leq a_n$ we have that $\abs{p_n'}\leq p_n$ and $\abs{q_n'}\leq q_n$. Therefore
\[
 \sup\left({\abs{A_{n-1}},\abs{A_{n}}}\right)\leq \sup\left(q_n^{-1},q_{n+1}^{-1}\right)=q_{n}^{-1},
\]
and
\[\sup\left({\abs{A_{n-1}'},\abs{A_{n}'}}\right)\leq \sup(\abs{\xi'}q_{n-1}+p_{n-1},\abs{\xi'}q_n+p_n)=\abs{\xi'}q_n+p_n.
\]
Hence
\begin{multline*}
 H(\xi_{n+1})=\sup\left({\abs{A_{n-1}},\abs{A_{n}}}\right)^{1/2}\sup\left({\abs{A_{n-1}'},\abs{A_{n}'}}\right)^{1/2}\prod_{v\text{ finite}}\sup\left({\abs{A_{n-1}}_v,\abs{A_{n}}_v}\right)\\
 \leq \left(\abs{\xi'}+\frac{p_n}{q_n}\right)^\frac{1}{2}\prod_{v\text{ finite}}\sup\left(\abs{\xi}_v,1\right)\leq (\abs{\xi'}+\abs{\xi}+1)^{1/2}\prod_{v\text{ finite}}\sup\left(\abs{\xi}_v,1\right)\\
 \leq \sqrt{3} \sup(\abs{\xi'},1)^{1/2}\sup(\abs{\xi},1)^{1/2}\prod_{v\text{ finite}}\sup\left(\abs{\xi}_v,1\right)=\sqrt{3}H(\xi),
\end{multline*}
where we used the elementary inequality $1+a+b\leq 3 \sup(1,a)\sup(1,b)$ for all $a,b\geq 0$.

Finally,
\[
H(a_n)=\sup\left(1,\abs{a_n}\right)^{1/2} \sup\left(1,\abs{a_n'}\right)^{1/2}\leq a_n < \xi_n \leq H(\xi_n)^2.
\]

Since the heights of the partial and complete quotients are bounded, Northcott's theorem (item~\eqref{northcott} of Proposition~\ref{prop:heights}) implies that they can take only finitely many values.
\end{proof}

The following proposition concerns the elements of the field $K$ which can be expressed as a purely periodic continued fraction. By an algebraic argument we show that, in order for the value of the continued fraction to lie in $K$, the numbers $p_n,q_n$ need to satisfy a certain arithmetic condition.

\begin{prop}\label{p:evperrac}
Let $K=\Q(\sqrt{D})$ be a real quadratic field.
 Let $\xi=[\overline{a_0,a_1,\dots, a_n}]$ be a purely periodic continued fraction with $a_i\in{\mathcal{O}_K}$. Assume that $\xi\in K$.
 Then $p_n+q_{n-1}\in\Z\cup\sqrt{D}\Z$.
\end{prop}

\begin{proof}
  Assume  that  $\xi=[\overline{a_0,\dotsc,a_n}]$ is a purely periodic continued fraction with partial quotients in ${\mathcal O}_K$ representing an element in $K$. Then by~\eqref{eq:raclom}, $\xi$ satisfies a quadratic equation over $K$, namely
 \[
  q_n \xi^2+(q_{n-1}-p_n)\xi-p_{n-1}=0.
 \]
The discriminant of this equation,
\[\Delta=(q_{n-1}-p_n)^2+4 q_n p_{n-1}=(q_{n-1}+p_n)^2+4(-1)^n=y^2,\]
must be a square of an element $y$ of $K$, because $\xi\in K$.

Let us write $x=p_n+q_{n-1}$, so we have
\begin{equation}\label{eq:disc.fatt}
 4(-1)^n =(y+x)(y-x),
\end{equation}
where $x,y\in\O_K$.

Let us show that this implies that $x'=\pm x$.

If the prime $2$ remains inert in $K$, then 2 must divide $y+x$ or $y-x$ by definition of a prime ideal. Then it must divide the other, because it divides their sum. So we have
\[
 (-1)^n=\frac{y+x}{2}\frac{y-x}{2}
\]
and both $\frac{y+x}{2},\frac{y-x}{2}$ are units in $\O_K$. Let us write $u=\frac{y+x}{2}$. Then $\frac{y-x}{2}=(-1)^n u^{-1}=(-1)^n \epsilon u'$, where $\epsilon=N_{K/\Q}(u)=\pm 1$. So $x=u-(-1)^n \epsilon u'$, which means that $x'=\pm x$.

If instead the prime $2$ ramifies as $2\O_K=\mathfrak{p}^2$, then as ideals $ \mathfrak{p}^4=(y+x)(y-x)$, so at least one of $y+x$ or $y-x$ must be divisible by $\mathfrak{p}^2=(2)$ and the argument proceeds as in the previous case.

If the prime $2$ splits as $2\O_K=\mathfrak{p}\mathfrak{p'}$, then as ideals $ \mathfrak{p}^2\mathfrak{p'}^2=(y+x)(y-x)$, so either one of $y+x$ or $y-x$ is divisible by $(2)$, and we argue again as above, or (without loss of generality) $\mathfrak{p}^2=(y+x)$ and $\mathfrak{p'}^2=(y-x)$ as ideals.
This means that there is a unit $u$ such that $y-x=u(y'+x')$. Substituting back in \eqref{eq:disc.fatt} we obtain $\pm 4=u N_{K/\Q}(x+y)$ which implies that $u\in\Q$, and so $u=\pm 1$. Now $y-uy'=x+ux'$ with $u=\pm 1$ implies that $x'=\pm x$. But clearly the elements of $\Z\cup\sqrt{D}\Z$ are the only numbers in $\mathcal{O}_K$ satisfying $x'=\pm x$.
\end{proof}

\begin{prop}\label{p:ZorSqrtDZ}
  Let $K=\Q(\sqrt{D})$ be a real quadratic field. Let $\xi=[\overline{a_0,\dotsc,a_n}]$ be a purely periodic continued fraction with $a_k\in \O_K$ and $a_k\geq 1$ for all $k\geq 0$. Assume that $\xi\in K$ and that  $\abs{a_k'}\leq {a_k}$ for  $k=0\dotsc,n$.

  Then either $a_k\in\Z$ for all $k=0,\dotsc,n$, or $a_k\in\sqrt{D}\Z$ for all $k=0,\dotsc,n$.
\end{prop}

\begin{proof}
 By Proposition~\ref{p:evperrac} we have that $p_n+q_{n-1}\in\Z\cup\sqrt{D}\Z$, which implies that $p_n+q_{n-1}=\abs{p'_n+q'_{n-1}}$.
 As explained in Section~\ref{subsec:CFintro}, $p_n+q_{n-1}=K_{n+1}(a_0,\dotsc,a_{n})+K_{n-1}(a_1,\dotsc,a_{n-1})$ is a polynomial with positive integer coefficients in the variables $a_0,\dotsc,a_n$; therefore \begin{align*}p_n+q_{n-1}=\abs{p'_n+q'_{n-1}}&\leq K_{n+1}(\abs{a'_0},\dotsc,\abs{a'_{n}})+K_{n-1}(\abs{a'_1},\dotsc,\abs{a'_{n-1}})\leq\\ &\leq K_{n+1}(a_0,\dotsc,a_{n})+K_{n-1}(a_1,\dotsc,a_{n-1})=p_n+q_{n-1}.\end{align*}
 This implies that equalities must hold at every point, so in particular $\abs{a'_k}=a_k$ for all $k$, which shows that each $a_k$ is in $\Z\cup\sqrt{D}\Z$.

Assume now that the period $\overline{a_0,\dots,a_n}$ contains both an $a_i\in\Z$ and an $a_j\in\sqrt{D}\Z$. Take the minimal $k\in\{0,\dots,n\}$ such that $a_ka_{k+1}\in\sqrt{D}\Z$.
Then by \eqref{eq:continuant}
$$
\begin{aligned}[t]
p_n&=K(a_0,\dots,a_k)K(a_{k+1},\dots,a_n)+K(a_0,\dots,a_{k-1})K(a_{k+2},\dots,a_n)=\\
&=\big(a_kK(a_0,\dots,a_{k-1})+K(a_0,\dots,a_{k-2})\big)\big(a_{k+1}K(a_{k+2},\dots,a_{n})+K(a_{k+3},\dots,a_{n})\big)\\
&\hspace*{6cm}+K(a_0,\dots,a_{k-1})K(a_{k+2},\dots,a_n)=\\
&=(a_ka_{k+1}+1)K(a_0,\dots,a_{k-1})K(a_{k+2},\dots,a_n)+K(a_0,\dots,a_{k-2})K(a_{k+3},\dots,a_n) + \\
&+ a_{k}K(a_0,\dots,a_{k-1})K(a_{k+3},\dots,a_n)+a_{k+1}K(a_0,\dots,a_{k-2})K(a_{k+2},\dots,a_n).
\end{aligned}
$$
Note that we have omitted the indices of the continuants since they are clear from the context.
By assumption, we have $a_k a_{k+1}+1\in\Z^++\sqrt{D}\Z^+$.
Since all the partial quotients are positive, we see that $x=p_n+q_{n-1}\in\Z^++\sqrt{D}\Z^+$.
  Then obviously, $\abs{x'}<x$, which is a contradiction.
\end{proof}

\begin{proof}[Proof of Theorem~\ref{t:superteorem}]
We are in the hypotheses of the second part of Lemma~\ref{l:completequobound}, so only finitely many distinct complete quotients occur in the continued fraction $[a_0,a_1,a_2,\dots]$. Note that their number can be effectively bounded in a way that depends only on the height of $\xi$ and not on the number field $K$, because the degree of $K$ over $\Q$ is fixed, see item (\ref{northcott}) of Proposition~\ref{prop:heights}. Necessarily at least one complete quotient  occurs infinitely many times. Consider any two occurrences of the same complete quotient, say $\xi_r=\xi_{r+s}$. Then we have
 \[
 \xi_r=[a_r,a_{r+1},\dotsc]=[a_r,a_{r+1},\dotsc,a_{r+s-1},\xi_{r+s}]=[\overline{a_r,a_{r+1},\dotsc,a_{r+s-1}}].
 \]
By Proposition~\ref{p:ZorSqrtDZ}, we have $a_{r+i}\in\Z$ for all $i=0,\dots,s-1$, or $a_{r+i}\in\sqrt{D}\Z$ for all $i=0,\dots,s-1$.

Since the same consideration can be done for any two occurrences of the complete quotient $\xi_r$, we derive that there exist $k_0$ such that $a_k\in\Z$ for all $k\geq k_0$, or $a_k\in\sqrt{D}\Z$ for all $k\geq k_0$.
In both cases, we have $\xi_r=[a_{k_0},a_{k_0+1},\dots]$, with partial quotients in a discrete set with distances between consecutive elements $\geq 1$. Such a sequence represents the number $\xi$ uniquely, see Lemma~\ref{l:unique}. Therefore the continued fraction of $\xi_r$ is purely periodic, and the continued fraction of $\xi$ is eventually periodic.
\end{proof}

\begin{remark}\label{rem:number.bounded.height}
A possible value for the constant $C_\xi$ of Theorem~\ref{t:superteorem} is given by $\#\{x\in K \mid H(x)\leq \sqrt{3}H(\xi)\}$.
 We remark that for a real quadratic field $K$ the cardinality of the finite set $\{x\in K \mid H(x)\leq B\}$ is asymptotic (for large $B$) to $c_K B^4$, where $c_K$ has an explicit expression in terms of the discriminant, regulator, number of ideal classes, number of roots of unity and Dedekind zeta function of the field $K$ (see \cite{schanuel}*{Corollary}).

A bound which does not depend on the field $K$ is given by
 \[\#\{x\in \Qbar \mid H(x)\leq B\text{ and }[\Q(x):\Q]=2\}\leq \frac{8}{\zeta(3)}B^6+16690B^4\log B\] for $B\geq \sqrt{2}$ (see \cite{slicing}*{Theorem 11.1}).

 We will present in the Appendix a different proof of Theorem~\ref{t:superteorem} in case that $|a'_n|<a_n$ which gives a better estimate for the number of irrational partial quotients.
\end{remark}

\section{\texorpdfstring{$M$}{M}-Continued Fractions}\label{sec:MCF}

In the classical theory of continued fractions the partial quotients are taken to be positive integers (except possibly the first one), and there is a canonical algorithm defined through the Gauss map that attaches to every $\xi\in\R$ a continued fraction converging to $\xi$.
We  define now a different expansion, using, instead of the classical integral part, the $M$-integral part for certain subsets $M$ of the real line.

Let $M=(m_n)_{n\in \Z}$ be an infinite subset of the reals without any accumulation points, enumerated in increasing order.
Assume that $m_0=0$, $m_1\leq 1$ and $m_{n+1}-m_n\leq 1/m_1$ for all $n\in \Z$.

For $\xi\in\R$ we define its $M$-integral part and its $M$-fractional part  as
$$
\lfloor \xi \rfloor _M := \max\{y\in M :y\leq \xi\},\qquad \{\xi\}_M:=\xi-\lfloor \xi\rfloor_M,
$$
respectively.

The $M$-continued fraction expansion of a number $\xi\in\R$ is now defined as a direct generalization of the regular continued fraction expansion obtained if $M=\Z$, as explained below.

\medskip

Let $\xi\in\R$. Set $\xi_0:=\xi$. For $n\geq 0$ define inductively:
\begin{equation}\label{eq:Manxin}
\begin{aligned}
a_n&=\lfloor \xi_n\rfloor_M, &\\
\xi_{n+1}&=\frac{1}{\xi_n-a_n} &\text{if }\xi_n\notin M,\\
p_{n}&=a_np_{n-1}+p_{n-2}, &p_{-1}=1,\ p_{-2}=0,\\
q_{n}&=a_nq_{n-1}+q_{n-2}, &q_{-1}=0,\ q_{-2}=1.
\end{aligned}
\end{equation}

If some $\xi_n\in M$, then the algorithm stops and we say that the $M$-continued fraction $[a_0,\dots,a_n]$ is finite.
Note that conditions on the set $M$ ensure that $a_n>0$ for $n\geq 1$.

\begin{remark}
In the classical case ($M=\Z$), it is possible to prove that any infinite sequence of positive integers is the continued fraction expansion of some real number. For a general set $M$ it might happen that certain sequences of positive partial  quotients in $M$ are not allowed, i.e.\ they never occur in the output of the iteration~\eqref{eq:Manxin} for any $\xi\in\R$. We will see an example later in Lemma~\ref{l:admis}.
\end{remark}

\begin{remark}\label{rem:ovvio}
  It is clear that if the $M$-continued fraction expansion of $\xi$ is finite, then $\xi\in\Q(M)$. If the continued fraction expansion of $\xi$ is eventually periodic, then $\xi$ is quadratic over $\Q(M)$.
\end{remark}

\section{\texorpdfstring{$\beta$}{Beta}-Integers and \texorpdfstring{$\beta$}{beta}-Continued Fractions}\label{sec:beta-integers}
\subsection{\texorpdfstring{$\beta$}{Beta}-integers}
In view of generalizing the result of~\cite{bernat}, one of the interests of the present paper is to describe properties of $M$-continued fraction expansions for a special class of sets $M$, the so-called $\beta$-integers, as defined in~\cite{BuFrGaKr}. Consider a real base $\beta>1$.
Any real $x$ can be expanded in the form $x=\pm \sum_{i=-\infty}^{k}x_i\beta^i$ where the digits $x_i\in\Z$ satisfy $0\leq x_i<\beta$. Under the condition that the inequality
\begin{equation}\label{eq:greedy}
\sum_{i=-\infty}^{j}x_i\beta^i < \beta^{j+1}
\end{equation}
holds for each $j\leq i$, we have that such a representation is unique up to the leading zeroes; this is called the greedy $\beta$-expansion of $x$, see R\'enyi~\cite{Renyi}. For the greedy expansion of $x$ we write
$$
(x)_\beta=x_kx_{k-1}\cdots x_0\bullet x_{-1}x_{-2}\cdots
$$
If $\beta\notin \N$, then not every sequence of digits in $\{k\in\N : 0\leq x<\beta\}$ corresponds to the greedy $\beta$-expansion of a real number.
Admissibility of digit sequences as $\beta$-expansions is described by the so-called Parry's condition~\cite{Parry}.

The $\beta$-expansions respect the natural order of real numbers in the radix ordering. In particular,
if $x=\sum_{i=-\infty}^{k}x_i\beta^i$ and $y=\sum_{i=-\infty}^{l}y_i\beta^i$ with $x_k,y_l\neq 0$ are the $\beta$-expansions of $x$, $y$, respectively, then
$x<y$ if and only if $k<l$, or $k=l$ and $x_kx_{k-1}\cdots$ is lexicographically smaller than $y_ky_{k-1}\cdots$.

A real number $x$ is called a $\beta$-integer if the greedy $\beta$-expansion of its absolute value $|x|$ uses only non-negative powers of the base $\beta$; we denote by $\Z_\beta$ the set of $\beta$-integers and by $\Z_\beta^{+}$ the set of non-negative $\beta$-integers, i.e.
$$
\begin{aligned}
\Z_\beta &=\{x\in\R : (|x|)_\beta=x_kx_{k-1}\cdots x_0\bullet 00\cdots\},\\
\Z_\beta^{+} &=\{x\in\R : (x)_\beta=x_kx_{k-1}\cdots x_0\bullet 000\cdots\}.
\end{aligned}
$$

If the base $\beta$ is in $\Z$, the $\beta$-integers are rational integers, i.e. $\Z_\beta=\Z$. Otherwise, it is an aperiodic set of points that can be ordered into a sequence $(t_j)_{j=-\infty}^\infty$, such that $t_i<t_{i+i}$ for $i\in\Z$ and $t_0=0$.
The smallest positive $\beta$-integers are
$$
1,\ 2,\dotsc, \lfloor\beta\rfloor,\ \beta,\ \dotsc
$$
We can derive the following property of $\beta$-integers.

\begin{lem}\label{l:perronhyp}
 Let $\beta>1$ be an algebraic number and let $\sigma$ be a Galois embedding of $\Q(\beta)$ in $\C$.
 \begin{enumerate}[(i)]
  \item\label{l:perronhyp:i} Assume that $\abs{\sigma(\beta)}<\beta$. Then for any $x\in\Z_\beta^+\setminus\{0,1,\dotsc,\floor{\beta}\}$ we have
   $$
   \frac{\abs{\sigma(x)}}{x}\leq \frac{\floor{\beta}+\abs{\sigma(\beta)}}{\floor{\beta}+\beta} <1.
   $$
   \item\label{l:perronhyp:ii} Assume that $\sigma(\beta)\in\R$ and $\sigma(\beta)>\beta$. Then for any $x\in\Z_\beta^+\setminus\{0,1,\dotsc,\floor{\beta}\}$ we have
      $$
      \frac{\sigma(x)}{x}\geq\frac{\floor{\beta}+\sigma(\beta)}{\floor{\beta}+\beta}>1.
      $$
      Moreover, if $\floor{\sigma(\beta)}>\floor{\beta}$, then for any pair $x,y\in\Z_\beta$, $x\neq y$, we have $|\sigma(x)-\sigma(y)|\geq 1$.

  \item\label{l:perronhyp:twoandhalf} Assume that $c=\frac{\abs{\sigma(\beta)}}{\beta}> 1+2\frac{\floor{\beta}}{\beta-1}$. Then for any $x\in\Z_\beta^+\setminus\{0,1,\dotsc,\floor{\beta}\}$ we have $$\frac{\abs{\sigma(x)}}{x}>\frac{c(\beta-1)-\floor{\beta}}{\beta-1+\floor{\beta}}>1.$$

   \item\label{l:perronhyp:iii} Assume that $\abs{\sigma(\beta)}=\beta$. Then for any $x\in\Z_\beta^+$ we have
   $\abs{\sigma(x)}\leq{x}$.

 \end{enumerate}
 In cases \eqref{l:perronhyp:i},\eqref{l:perronhyp:ii}, and \eqref{l:perronhyp:twoandhalf} we have that $\Z_\beta^+ \cap \Q=\{0,1,\dotsc,\floor{\beta}\}$.
\end{lem}

\begin{proof}
Let $x=\sum_{i=0}^{k}x_i\beta^i\in\Z_\beta^+\setminus\{0,1,\dotsc,\floor{\beta}\}$, so that $x_0\leq \floor{\beta}$ and $\sum_{i=1}^{k}x_i\beta^i\geq \beta$.
Let us prove the first item \eqref{l:perronhyp:i}.
 Denote $c:= \abs{\sigma(\beta)}/\beta<1$.
 Then for the algebraic conjugate $\sigma(x)$ of $x$ we have
 \begin{multline*}
  \frac{\abs{\sigma(x)}}{x}= \frac{\abs{x_0+\sum_{i=1}^{k}x_i\sigma(\beta)^i}}{x_0+\sum_{i=1}^{k}x_i\beta^i}\leq \frac{x_0+\sum_{i=1}^{k}x_i c^i \beta^i}{x_0+\sum_{i=1}^{k}x_i\beta^i}\leq\\
  \leq \frac{x_0+c\sum_{i=1}^{k}x_i \beta^i}{x_0+\sum_{i=1}^{k}x_i\beta^i}\leq  \frac{\floor{\beta}+c\beta}{\floor{\beta}+\beta}=\frac{\floor{\beta}+\abs{\sigma(\beta)}}{\floor{\beta}+\beta}<1.
 \end{multline*}

 Very similarly, for proving part \eqref{l:perronhyp:ii} denote $c:= \sigma(\beta)/\beta>1$.
 Then for the algebraic conjugate $\sigma(x)$ of $x$ we have
 \begin{multline*}
  \frac{\sigma(x)}{x}= \frac{x_0+\sum_{i=1}^{k}x_i\sigma(\beta)^i}{x_0+\sum_{i=1}^{k}x_i\beta^i}= \frac{x_0+\sum_{i=1}^{k}x_i c^i \beta^i}{x_0+\sum_{i=1}^{k}x_i\beta^i}\geq\\
  \geq  \frac{x_0+c\sum_{i=1}^{k}x_i \beta^i}{x_0+\sum_{i=1}^{k}x_i\beta^i}\geq
  \frac{\floor{\beta}+c\beta}{\floor{\beta}+\beta}=\frac{\floor{\beta}+\sigma(\beta)}{\floor{\beta}+\beta}>1.
 \end{multline*}

Suppose now that $\floor{\sigma(\beta)}=M>m=\floor{\beta}$. We shall estimate the distance $|\sigma(x)-\sigma(y)|$ of two
distinct $\beta$-integers $x=\sum_{i=0}^{k}x_i\beta^i,y=\sum_{i=0}^{l}y_i\beta^i$. If $k=l=0$, obviously $|\sigma(x)-\sigma(y)|\geq 1$. Otherwise, we can without loss of generality assume that $k>l$, $x_k>0$. We have
\begin{equation}\label{eq:vetsinezjedna?}
|\sigma(x)-\sigma(y)| = \Big| \sum_{i=0}^{k}x_i\sigma(\beta)^i - \sum_{i=0}^{l}y_i\sigma(\beta)^i \Big| \geq \sigma(\beta)^k - \sum_{i=0}^{k-1} m\sigma(\beta)^i\,.
\end{equation}
Here we have used that $m=\floor{\beta}$ is the maximal digit allowed in the $\beta$-expansion. Now consider the expansion in base $\sigma(\beta)$. The maximal allowed digit is $M=\floor{\sigma(\beta)}>m$. It can be easily derived from the Parry condition~\cite{Parry} that $\sum_{i=1}^{k-1} m\sigma(\beta)^i+(m+1)$ is an admissible greedy $\sigma(\beta)$-expansion. By definition~\eqref{eq:greedy}, we have
$$
\sum_{i=1}^{k-1} m\sigma(\beta)^i+(m+1) < \sigma(\beta)^k,
$$
which is equivalent to
$$
\sigma(\beta)^k - \sum_{i=0}^{k-1} m\sigma(\beta)^i > 1.
$$
Comparing with~\eqref{eq:vetsinezjedna?} we have the statement of item (ii).

And again for part \eqref{l:perronhyp:twoandhalf} we have
 \begin{multline*}
  \frac{\abs{\sigma(x)}}{x}\geq \frac{x_k (c\beta)^k -\abs{\sum_{i=0}^{k-1}x_i\sigma(\beta)^i}}{x_k\beta^k+\sum_{i=0}^{k-1}x_i\beta^i} \geq \frac{ (c\beta)^k -\abs{\sum_{i=0}^{k-1}x_i\sigma(\beta)^i}}{\beta^k+\sum_{i=0}^{k-1}x_i\beta^i}\geq\\
  \geq \frac{(c\beta)^k -\sum_{i=0}^{k-1}x_i c^i \beta^i}{\beta^k+\sum_{i=0}^{k-1}x_i\beta^i}\geq  c^{k-1} \frac{ c \beta^k-\sum_{i=0}^{k-1}x_i \beta^i}{\beta^k+\sum_{i=0}^{k-1}x_i\beta^i}\geq\\
  \geq\frac{ c \beta^k-\floor{\beta}\frac{\beta^k-1}{\beta-1}}{\beta^k+\floor{\beta}\frac{\beta^k-1}{\beta-1}}>
  \frac{ c -\frac{\floor{\beta}}{\beta-1}}{1+\frac{\floor{\beta}}{\beta-1}}=
  \frac{c(\beta-1)-\floor{\beta}}{\beta-1+\floor{\beta}}>1.
 \end{multline*}
The proof of \eqref{l:perronhyp:iii} is simple realizing that
$$
\abs{\sigma(x)}=\Big|\sum_{i=0}^{k}x_i\sigma(\beta)^i\Big|\leq \sum_{i=0}^{k}x_i\abs{\sigma(\beta)}^i = x\,.
$$

In order to show that in cases (i) and (ii) it holds that $\Z_\beta^+ \cap \Q=\{0,1,\dotsc,\floor{\beta}\}$, realize that in these cases any $x\in\Z_\beta^+\setminus\{0,1,\dotsc,\floor{\beta}\}$ is not fixed by a Galois automorphism, and therefore it cannot be rational.
\end{proof}

\begin{remark}
 Notice that the final assertion of the lemma, namely that $\Z_\beta^+ \cap \Q=\{0,1,\dotsc,\floor{\beta}\}$,  does not hold in general. Obvious counterexamples are given by $\beta$'s which are square roots of rational integers, in which case all even powers of $\beta$ are in $\Z_\beta^+\cap \Q$.
 Less trivial counterexamples can occur for some quadratic $\beta$'s with  $\beta'<-\beta$.
 For instance if $\beta$ is the positive root of $X^2+2X-9=0$, then $\beta^2+2\beta=9\in\Z_\beta^+\cap \Q$.

 More in general, if $\beta$ is the positive root of $X^2+mX-(2m^2+1)$, for $m\geq 2$, then it can be shown that $\floor{\beta}=m$ and that $\beta^2+m\beta=2m^2+1\in\Z_\beta^+\cap \Q$.
\end{remark}

A class of algebraic numbers satisfying the assumptions of Lemma~\ref{l:perronhyp}~\eqref{l:perronhyp:i}, is formed by Perron numbers, i.e.\ algebraic integers $\beta>1$ such that every conjugate $\sigma(\beta)$ of $\beta$ satisfies $\abs{\sigma(\beta)}<\beta$. Perron numbers appear as dominant eigenvalues of primitive integer matrices. A special subclass of Perron numbers are Pisot numbers, i.e.\ algebraic integers $\beta>1$ whose conjugates lie in the interior of the unit disc.
Note that when $\beta$ is chosen to be a Pisot number, then the Galois conjugates of the $\beta$-integers are uniformly bounded and consequently, the $\beta$-integers enjoy many interesting properties, especially from the arithmetical point of view, see e.g.~\cite{FroSo} or~\cite{AkiyamaPisotAndGreedy}.
We also mention that Perron numbers in any given number field form a discrete subset.
\subsection{\texorpdfstring{$\beta$}{Beta}-continued fractions}
When setting $M$ to be the set $\Z_\beta$ of $\beta$-integers, it is clear that any finite $\beta$-continued fraction belongs to the field $\Q(\beta)$. The opposite is however not obvious as we will see in sequel. For that, we define properties \CFF{} and \CFP.

\begin{de}
For a real number $\beta>1$ set $M=\Z_\beta$. The $M$-continued fraction in this case is said to be the $\beta$-continued fraction.
 We say that a real number $\beta>1$ has the Continued Fraction Finiteness property \CFF{} if every element of $\Q(\beta)$ has a finite $\beta$-continued fraction expansion.
 We say that $\beta>1$ has the Continued Fraction Periodicity property \CFP{} if every element of $\Q(\beta)$ has either finite or eventually periodic $\beta$-continued fraction expansion.
\end{de}

Note that properties \CFF{} and \CFP{} could be studied also for other sets $M$.

\begin{remark}
 If $\beta\in\N$, then $\Z_\beta=\Z$ and the $\beta$-continued fraction expansion is the classical regular continued fraction expansion; this implies that integer $\beta$'s satisfy \CFF{}.
\end{remark}

\begin{remark}\label{r:ZCF-betaCF}
 Let $\beta>1$ and $\xi>0$. If the regular continued fraction expansion of $\xi$ only involves partial quotients strictly smaller than $\floor{\beta}$, then it coincides with the $\beta$-continued fraction expansion of $\xi$.
\end{remark}

Based on the above remark, any sequence of integers in $\{1,2,\dots,\floor{\beta}-1\}$ is a $\beta$-continued fraction expansion of a real number $x$. In general, however, not any sequence of $\beta$-integers will occur as some $\beta$-continued fraction expansion. The following statement describes a sequence of partial quotients in $\Z_\beta$ which is not admissible in a $\beta$-continued fraction expansion.

\begin{lem}\label{l:admis}
  Let $\beta>1$ satisfy $\beta-\lfloor\beta\rfloor \leq \beta^{-1}$. Let $a_i,a_{i+1}$ be two consecutive partial quotients in the $\beta$-continued fraction expansion of a real number $\xi$. If $a_i=\lfloor\beta\rfloor$, then
   $a_{i+1}\geq \beta$.
\end{lem}

\pfz
If $a_i=\lfloor \xi_i\rfloor_{\Z_\beta}=\lfloor\beta\rfloor$, then $\lfloor\beta\rfloor< \xi_i<\beta$. Thus $\xi_{i+1}=(\xi_i-\lfloor\beta\rfloor)^{-1}>(\beta-\lfloor\beta\rfloor)^{-1}= \beta$.
\pfk

In what follows we will study property \CFF{} in quadratic fields.
For quadratic Pisot units $\beta$, the set of all rules for admissibility of strings of partial quotients is given in~\cite{feng}. Let us mention that similar study can be found already in~\cite{Kolar}. In Lemma~\ref{l:admis} we have cited only the rule which will be needed later.

\section{Properties of \texorpdfstring{$\beta$}{Beta}-Continued Fractions for Quadratic Numbers}

We aim to characterize which quadratic integers have properties $\CFF$ and $\CFP$. We will apply the general results of Section~\ref{sec:MCF_in_quadratic} to the set $M=\Z_\beta$ defined above.
In view of Remark~\ref{r:ZCF-betaCF}, it is clear that results on the classical regular continued fractions can have implications on $\beta$-continued fractions for $\beta$ big enough. In this spirit, the following proposition shows that \CFF{} property among quadratic numbers is rather rare.

\begin{prop}\label{p:CFFonlyFinitely}
 For every real quadratic field $K$ there is a positive bound $m_K>1$ such that no irrational number $\beta>m_K$ in $K$ has property \CFF{}.
 In particular in a given real quadratic field $K$ only finitely many Perron numbers can have property \CFF{}.
\end{prop}

\begin{proof}
 Let $\xi\in K\setminus\Q$. The regular continued fraction expansion of $\xi$ is eventually periodic. Denote by
 $c$ the maximum of the partial quotients appearing in the periodic part.
  Let $\beta>c+1$ be an element of $\O_K\setminus\Z$. Then $\beta$ does not have \CFF{}, because by Remark~\ref{r:ZCF-betaCF} the complete quotient of $\xi$ defined by the purely periodic tail does not have a finite $\beta$-continued fraction expansion. We can thus take $m_K=c+1$.
\end{proof}

The matter of studying continued fractions with small partial quotients is already very complicated for the classical regular continued fractions.
In \cite{McMullen}, where he studies the issue under the point of view of dynamics and geodesics on arithmetic manifolds, McMullen poses the following problem:

\begin{problem*}[\cite{McMullen}*{p.22}]
Does every real quadratic field contain infinitely many periodic continued
fractions with partial quotients equal to 1 or 2?
\end{problem*}

Mercat conjectures an affirmative answer to a weaker version of this problem:
\begin{conj}[\cite{mercat}*{Conjecture 1.6}]\label{conj:mercat}
 Every real quadratic field contains a periodic continued fraction with partial quotients equal to 1 or 2.
\end{conj}

\begin{remark}\label{rem:mercat}
 If $K=\Q(\sqrt{d})$, then standard estimates on the continued fraction expansion of
 $\big\lfloor\sqrt{d}\big\rfloor+\sqrt{d}$ imply that for the constant $m_K$ of Proposition~\ref{p:CFFonlyFinitely}, we can take $m_K = 2\big\lfloor\sqrt{d}\big\rfloor+1$, see~\cite{mercat}*{Proposition 7.11}.

 If Mercat's Conjecture is true, we can take $m_K = 3$ for all quadratic fields $K$.
\end{remark}

\subsection{Pisot numbers and Perron numbers}

Bernat~\cite{bernat} showed that the golden ratio $\varphi=\frac12(1+\sqrt{5})$ has property \CFF{} and asked whether other quadratic Pisot numbers with \CFF{} could be found.
The quadratic Pisot numbers can be characterized as the larger roots of one of the polynomials with integer coefficients
$$
X^2-aX-b,\ a\geq b\geq 1,\qquad X^2-aX+b,\ a\geq b+2\geq 3.
$$
The golden ratio $\varphi$ with minimal polynomial $x^2-x-1$ is the smallest among quadratic Pisot numbers.

We consider the more general quadratic Perron numbers. It is not difficult to see that these are the larger roots of the polynomials
$$
X^2-aX-b,\ a\geq 1,\ \text{ such that } a^2+4b>0,\ \sqrt{a^2+4b}\notin\Q.
$$
In view of Remark~\ref{rem:mercat}, we will focus on quadratic Perron numbers smaller than $3$.


An intermediate step towards establishing property \CFF{} is formulated in the following statement, which is a special case of Theorem~\ref{t:superteorem}.

\begin{thm}\label{t:quadraticperiodic}
Let $\beta$ be a quadratic Perron number. Then the $\beta$-continued fraction of any $\xi\in \Q(\beta)$ contains at most finitely many partial quotients in $\Z_\beta\setminus\Z$ and thus
it is either finite or eventually periodic with all partial quotients in the period being rational integers.

In particular, quadratic Perron numbers satisfy $\CFP$.
\end{thm}

\begin{proof}
It suffices to check that the assumptions of Theorem~\ref{t:superteorem} are satisfied if the $\beta$-continued fraction expansion of $\xi$ is not finite.
This is shown by item (i) of Lemma~\ref{l:perronhyp}. By Theorem~\ref{t:superteorem}, the partial quotients in the period belong to $\Z$ or $\sqrt{D}\Z$. The second possibility
is however not possible, again by item (i) of Lemma~\ref{l:perronhyp}.
\end{proof}

We apply Theorem~\ref{t:quadraticperiodic} to establish \CFF{} for the four smallest quadratic Perron numbers. Note that Bernat~\cite{bernat} showed the statement for $\varphi$ by completely different methods.

\begin{thm}\label{t:3perron}
 The four Perron numbers
 \begin{align*}
  \varphi=\frac{1+\sqrt5}{2},& &1+\sqrt2, & & \frac{1+\sqrt{13}}{2}, & &\frac{1+\sqrt{17}}{2}
 \end{align*}
 have property \CFF{}.
\end{thm}

\begin{proof}
Let $\beta$ be one of the four Perron numbers above.
From Theorem~\ref{t:quadraticperiodic} we derive that the $\beta$-continued fraction expansion of any $\xi\in \Q(\beta)$ is either finite, or eventually periodic, with the partial quotients appearing in the period being integers
smaller than $\beta$.

If $\beta=\varphi$, then the only possible periodic tail is $[\overline{1}]$, but by Lemma~\ref{l:admis} it is not an admissible $\varphi$-continued fraction expansion.

If $\beta$ is one of the other three values appearing in the statement, then $2<\beta<\frac{18}{7}$ and the partial quotients in a periodic tail belong necessarily to $\{1,2\}$.

Assume that the period contains at least one partial quotient equal to 1 and one equal to 2. Up to replacing $\xi$ by one of its complete quotients, we can assume that the $\beta$-continued fraction expansion of $\xi$ begins with $[2,1,a_2,\dotsc]$, so that, writing $\xi_3$ for the third complete quotient, we have
\[
 \xi=[2,1,a_2,\xi_3]=2+\cfrac{1}{1+\cfrac{1}{a_2+\cfrac{1}{\xi_3}}}=\frac{3 a_2 \xi_3+2\xi_3+3}{a_2 \xi_3+\xi_3+1},
\]
where $a_2\in\{1,2\}$ and $1<\xi_3<3$.
This rational expression is easily seen to be strictly increasing with $a_2$ and strictly decreasing with $\xi_3$, so we have that $\xi>\frac{18}{7}>\beta$; this is a contradiction, because then the expansion would not start with a 2.

The only possible periodic tails are therefore $[\overline{1}]=\frac{1+\sqrt{5}}{2}$ and $[\overline{2}]=1+\sqrt{2}$. The first possibility is excluded because $\varphi\not\in\Q(\beta)$; for the same reason, the second possibility could only occur if $\beta=1+\sqrt{2}$, and in this case Lemma~\ref{l:admis} again shows that $[\overline{2}]$ cannot be a $\beta$-continued fraction expansion.

We conclude that every $\xi\in\Q(\beta)$ has a finite $\beta$-continued fraction expansion.
\end{proof}

\begin{remark}
Notice that the first three numbers in the statement of Theorem~\ref{t:3perron} satisfy the hypothesis of Lemma~\ref{l:admis}. This allows, in the proof of the theorem, to argue that the periodic tails cannot contain any partial quotient equal to 2, which leads to a quicker conclusion. However, this argument does not work for $\beta=\frac{1+\sqrt{17}}{2}$. For example $$\frac{164+65\sqrt{17}}{251}=[\overline{1,1,2,1,1,2,2,2,2}]$$ is a number whose classical continued fraction expansion is periodic and uses only partial quotients equal to 1 and 2; but this is not its $\beta$-expansion, which is
$$\frac{164+65\sqrt{17}}{251}=[1,1,\beta,2\beta^3+\beta^2+1,\beta^3+\beta+1,2,\beta+1].$$
\end{remark}

\begin{coro}
 The Perron numbers
 \begin{align*}
  \varphi=\frac{1+\sqrt5}{2},& &1+\sqrt2, & & \frac{1+\sqrt{13}}{2}, & &\frac{1+\sqrt{17}}{2}
 \end{align*}
 are the only quadratic Perron numbers smaller than 3 having property \CFF{}.

 Assuming Mercat's Conjecture~\ref{conj:mercat}, they are the only quadratic Perron numbers with property \CFF{}.
\end{coro}

\begin{proof}
 Table~\ref{table} gives a full list of quadratic Perron numbers smaller than 3.
\begin{table}[h]
  \centering
  {\setlength{\tabcolsep}{5pt}
\renewcommand{\arraystretch}{1.3}
\begin{tabular}{|c|c|c|c|c|}
  \hline
  $\beta$ & Approximate value & Minimal polynomial & Pisot unit & \CFF{}\\ \hline
  $\frac{1}{2}(1+\sqrt{5})$ & 1.618033988... & $x^2-x-1$ & yes & yes\\
  $\frac{1}{2}(1+\sqrt{13})$ & 2.302775637... & $x^2-x-3$ & no & yes\\
  $1+\sqrt{2}$ & 2.414213562... & $x^2-2x-1$ & yes & yes \\
  $\frac{1}{2}(1+\sqrt{17})$ & 2.561552812... & $x^2-x-4$  & no & yes \\
  $\frac{1}{2}(3+\sqrt{5})$ & 2.618033988... & $x^2-3x+1$ & yes & no \\
  $1+\sqrt{3}$ & 2.732050807... & $x^2-2x-2$ & no & no\\
  $\frac{1}{2}(1+\sqrt{21})$ & 2.791287847... & $x^2-x-5$  & no & no \\
  \hline
\end{tabular}
\medskip
}
  \caption{Quadratic Perron numbers smaller than 3.}\label{table}
\end{table}

Theorem~\ref{t:3perron} shows that the first four of them have property \CFF{}.
The following counterexamples, which can be respectively shown to be $\beta$-continued fraction expansions for the last three values of $\beta$ in the list, show that these values do not have property \CFF{}:
\begin{align*}
[\overline{1}] &=\frac{1+\sqrt{5}}{2},\\
 [\overline{1,1,1,1,1,1,1,2,1,1,2,1,2,1,1,1,2,2}] &=\frac{11055+10864\sqrt{3}}{18471},\\
[\overline{1,1,1,2,1,2,1,2,2,2,1,1,2,2}] &= \frac{117+44\sqrt{21}}{202}.
 \end{align*}
According to Remark~\ref{rem:mercat}, under Mercat's Conjecture~\ref{conj:mercat} no other quadratic Perron number can have property \CFF{}.
\end{proof}


\begin{remark}
Refuting \CFF{} for quadratic integers bigger than 3 depends on the validity of Mercat's Conjecture. However, for the subclass of quadratic Pisot units, we can provide explicit examples of bases $\beta$ for which we can disprove property \CFF{}. In this way we are able to determine unconditionally that the only quadratic Pisot units with property \CFF{} are $\frac12(1+\sqrt{5})$ and $1+\sqrt{2}$. The examples of infinite $\beta$-continued fractions in the field $\Q(\beta)$ for other quadratic Pisot units $\beta$ are the following:

\begin{itemize}
  \item $\beta>1$, root of $X^2-mX-1$, $m\geq 3$ :

  \begin{itemize}
    \item $m$ even :
    $$
    [\overline{(m-2)/2, 1, 1}] = \frac{m-2+\sqrt{m^2+4}}{4}=\frac{\beta-1}{2}\in\Q(\beta),
    $$
    \item $m$ odd, $m\geq 5$ :
    $$
    [\overline{(m-3)/2,(m+1)/2,3,1}] = \frac{m^2-3m-m+m\sqrt{m^2+4}}{4m+6}\in\Q(\beta),
    $$
    \item $m=3$ :
    $$
    [\overline{1,1,2,2,2}]=\frac{11+5\sqrt{13}}{17}\in\Q(\sqrt{13}).
    $$
  \end{itemize}

  \item $\beta>1$, root of $X^2-mX+1$, $m\geq 3$ :
  $$
  [\overline{1,m-2}]=\frac{m-2+\sqrt{m^2-4}}{2(m-2)}\in\Q(\beta).
  $$
\end{itemize}

This last example was already given in~\cite{Kolar}.
\end{remark}

\subsection{Square roots}\label{subsec:sqrt}

The simplest example of a non-Perron quadratic integer is an irrational square root of a rational integer $D$, i.e.\ $\beta=\sqrt{D}\notin\Q$. Such $\beta$ satisfies $\beta'=-\beta$. For the study of finiteness and periodicity of $\beta$-continued fractions in this case we apply results of Section~\ref{sec:MCF_in_quadratic}.

\begin{thm}\label{t:squareroots}
  Let $D$ be a positive integer with irrational square root. Denote $\beta=\sqrt{D}$. Then $\beta$ satisfies $\CFP$ and, under Mercat's conjecture, does not satisfy $\CFF$.
\end{thm}

\begin{proof}
The fact that $\beta=\sqrt{D}$ satisfies $\CFP$ follows from Theorem~\ref{t:superteorem} with the use of Lemma~\ref{l:perronhyp} item \eqref{l:perronhyp:iii}. Let us now focus on property $\CFF$.
Assuming the Mercat's conjecture, we only need to check $\beta\in\{\sqrt2,\sqrt3,\sqrt5,\sqrt6,\sqrt7,\sqrt8\}$. To disprove $\CFF$ in these cases it suffices to consider periodic $\beta$-continued fractions
\begin{align*}
 [\overline{4\sqrt{2}}] &=3+2\sqrt{2}, & [\overline{8\sqrt{6},2\sqrt{6}}] &=2(5+2\sqrt{6}),\\
 [\overline{3,4}] &=\frac{3+2\sqrt{3}}{2}, & [\overline{\sqrt{7},2\sqrt{7}}] &=\frac{3+\sqrt{7}}{2},\\
 [\overline{32\sqrt{5},2\sqrt{5}}] &= 16(9+\sqrt{5}), & [\overline{2\sqrt{8}}] &= 3+\sqrt{8}.\qedhere
\end{align*}
\end{proof}

\subsection{Non-Perron quadratic integers with positive conjugate}\label{subsec:nonperronpos}

Consider now as base $\beta$ a quadratic integer which is not Perron, so that Theorem~\ref{t:superteorem} may no longer be applied. We are nevertheless able to obtain information on the $\beta$-continued fractions by other means. We will show that neither $\CFF$ nor $\CFP$ hold by showing the existence of elements in the quadratic field $\Q(\beta)$ with aperiodic infinite $\beta$-continued fraction expansion.

\begin{thm}\label{t:nonperronklad}
 Let $\beta>1$ be a quadratic integer with conjugate $\beta'$ satisfying $\beta'>\beta$. Then each $\xi\in\Q(\beta)$, such that $\xi>\beta$ and $\xi'\in(-1,0)$, has an aperiodic $\beta$-continued fraction expansion.

In particular, $\beta$ satisfies neither $\CFF$ nor $\CFP$.
\end{thm}

The first step towards the proof is an application of the algebraic argument in Proposition~\ref{p:evperrac}.

\begin{prop}\label{p:evperracklad}
Let $\beta>1$ be a quadratic integer such that $\beta'>\beta$.
 Let $\xi$ be an element in $\Q(\beta)$ whose $\beta$-continued fraction expansion is eventually periodic. Then the period consist only of partial quotients in $\{1,\dotsc,\floor{\beta}\}$.
\end{prop}

\begin{proof}
  Assume (replacing it by a complete quotient if needed) that  $\xi=[\overline{a_0,\dotsc,a_n}]$ is the purely periodic $\beta$-continued fraction expansion of an element in $\Q(\beta)$. Then by Proposition~\ref{p:evperrac}, for $x=p_n+q_{n-1}$ we have that $x'=\pm x$.

Recall now that $p_n$, $q_{n-1}$ arise from the continuants, and thus also $x$ is a polynomial with positive coefficients in the partial quotients $a_0,\dotsc,a_n$. By item~\eqref{l:perronhyp:ii} of Lemma~\ref{l:perronhyp}, for each partial quotient $a_i$ in $\Z_\beta^+\setminus\{1,\dotsc,\floor{\beta}\}$, its image $a'_i$ satisfies $a'_i>a_i$.
We obtain that $x'\neq \pm x$ unless all partial quotients appearing in the period belong to $\{1,\dotsc,\floor{\beta}\}$, thus proving the statement.
\end{proof}

\begin{remark}\label{rem:vzdalenostbeta-beta'}
  Let $\beta$ be a quadratic integer. Notice that $\abs{\beta-\beta'}$ is at least one, because it is equal to the absolute value of the square root of the discriminant of the minimal polynomial of $\beta$.
\end{remark}

\begin{proof}[Proof of Theorem~\ref{t:nonperronklad}]
First notice that if $\xi_i' \in (-1,0)$, then $\xi_{i+1}'=\tfrac1{\xi_i'-a_i'}\in(-1,0)$ because $a_i'\geq1$. Thus the $\beta$-continued fraction expansion of any $\xi\in \Q(\beta)$ with $\xi'\in(-1,0)$ is infinite.

Next we prove that if $\xi$ has eventually periodic $\beta$-continued fraction, then it is in fact purely periodic. Then, by Proposition~\ref{p:evperracklad}, all its partial quotients belong to $\{1,\dots,\floor{\beta}\}$. However, $a_0=\floor{\xi}_\beta\geq \beta$, which gives a contradiction.

To show pure periodicity of the $\beta$-continued fraction of $\xi$ we argue as in the characterization of purely periodic classical continued fractions. Assume that $\xi$ has an eventually periodic $\beta$-continued fraction expansion. Take $k<l$ to be the minimal indices such that $\xi_k=\xi_l$. Then we either have $k=0$ and the proof is finished, or
$k>0$ and we have $\xi'_{k-1}=a'_{k-1}+\frac1{\xi'_k}\in(-1,0)$, which implies
\begin{equation}\label{eq:aunique}
-1+\frac1{\xi'_k}<a'_{k-1}<\frac1{\xi'_k}.
\end{equation}
By Remark~\ref{rem:vzdalenostbeta-beta'}, we have $\floor{\beta'}>\floor{\beta}$, and therefore by item (ii) of Lemma~\ref{l:perronhyp}, the distances between the conjugates $x',y'$ of $\beta$-integers $x,y$ are at least one. Consequently, \eqref{eq:aunique} defines $a'_{k-1}=a'_{l-1}$ uniquely.
Thus
$$
\xi_{k-1}=a_{k-1}+\frac1{\xi_k} = a_{l-1}+\frac1{\xi_l} = \xi_{l-1}\,,
$$
which contradicts the minimality of indices $k,l$.
\end{proof}

\begin{remark}
  Let us mention that in refuting $\CFF$ for $\beta$ with $\beta'>\beta$ we did not use Mercat's conjecture.
   Another interesting fact to mention is that the infinitely many elements of $\Q(\beta)$ with aperiodic $\beta$-continued fraction expansion found in Theorem~\ref{t:nonperronklad} belong to the family of numbers with purely periodic
   regular continued fraction.
\end{remark}

\subsection{Non-Perron quadratic integers with negative conjugate}

It remains to treat the case when $\beta>1$ is a quadratic integer such that $\beta'<-\beta$. These $\beta$'s are the positive irrational roots of the polynomials of the shape
\begin{equation}\label{eq:betanegatconjug}
X^2+bX-c,\quad\text{where $b\geq 1$ and $c\geq b+2$ are integers.}
\end{equation}
Notice that in this case $\beta'=-\beta-b$.
Here the situation appears to be more complicated than in the other cases already treated. We shall prove the following statement.

\begin{thm}\label{t:negatconjugnotCFF}
Let $\beta>1$ be a root of~\eqref{eq:betanegatconjug} with $b\geq 4$, i.e.\ $\beta$ is a
quadratic integer with conjugate $\beta'$ satisfying $\beta'\leq -\beta-4$. Let $\xi\in\Q(\beta)$ be such that $\xi'\in(-1,0)$. Then $\xi$ does not have a finite $\beta$-expansion.

In particular, $\beta$ does not satisfy property $\CFF$.
\end{thm}

\begin{lem}\label{lemma:inequalitiesk}
Let $b\geq 1$ be integers and assume that $\beta'= -\beta-b$. Then the function
\[
F(k)=\beta'^{2k+1}+\floor{\beta}\sum_{i=0}^k \beta'^{2i},    \quad k\geq 0,
\]
is decreasing with $k$, and the function
\[
G(k)=\beta'^{2k}+\floor{\beta}\sum_{i=0}^{k-1} \beta'^{2i+1},    \quad k\geq 1,
\]
is increasing with $k$.
\end{lem}

\begin{proof}
We see that
\begin{align*}
F(k+1)-F(k)=\beta'^{2k+1}\left(\beta'^2-1+\floor{\beta}\beta'\right)\\
G(k+1)-G(k)=\beta'^{2k}\left(\beta'^2-1+\floor{\beta}\beta'\right).
\end{align*}
But $\floor{\beta}\leq\beta$, which implies that
\[
\beta'^2+\floor{\beta}\beta'-1\geq \beta'^2+\beta\beta'-1= (\beta+b)b-1> 1,
\]
which completes the proof.
\end{proof}

\begin{prop}
Assume that $\beta'\leq-\beta-4$.
Then for every non-zero $x\in\Z_\beta^+$ we have that either $x=1$, or $x'\geq 2$, or $x'< -4$.
\end{prop}

\begin{proof}
Write $\beta'=-\beta-b$, for $b\geq 4$ an integer.
Let $x=\sum_{i=0}^{n}x_i\beta^i\in\Z_\beta^+$. Assume first that the highest power of $\beta$ appearing in $x$ is even. If this power is 0 then $x\in\{1,\dotsc,\floor{\beta}\}$, so $x'=1$ or $x'\geq 2$. Otherwise, write $n=2k$ with $k\geq 1$.
Then
\begin{equation*}
x'=\sum_{i=0}^{n}x_i\beta'^i\geq \beta'^{2k}+\floor{\beta}\sum_{i=0}^{k-1}\beta'^{2i+1}\geq \beta'^{2}+\floor{\beta}\beta' \geq \beta'(\beta'+\beta)=(\beta+b)b\geq 2,
\end{equation*}
where the second inequality follows from Lemma~\ref{lemma:inequalitiesk}.
Similarly, if the highest power appearing in $x$ is $n=2k+1$ odd, we have
\begin{equation*}
x'=\sum_{i=0}^{n}x_i\beta'^i\leq \beta'^{2k+1}+\floor{\beta}\sum_{i=0}^{k}\beta'^{2i}\leq \beta'+\floor{\beta}=\floor{\beta}-\beta-b<-b\leq-4.\qedhere
\end{equation*}
\end{proof}

\begin{proof}[Proof of Theorem \ref{t:negatconjugnotCFF}]
By the previous proposition, we have that the conjugate of any partial quotient is either equal to $1$, or is at least $2$, or at most $-4$. This allows us to track the position of the conjugates of the complete quotients, and to check that they never become zero. From the relation $\xi_{k+1}=1/(\xi_k-a_k)$, we can derive the following list of implications which is graphically represented in Figure~\ref{fig:automat}.

\begin{enumerate}
  \item if $\xi_k'\in(-1,0)$, then
  \begin{enumerate}
    \item $a_k'\geq1\implies \xi_{k+1}'\in(-1,0)$
    \item $a_k'\leq -3\implies \xi_{k+1}'\in(0,\tfrac12)$
  \end{enumerate}
  \item if $\xi_k'\in(0,\tfrac12)$, then
    \begin{enumerate}
	\item $a_k'\geq \tfrac32\implies \xi_{k+1}'\in(-1,0)$
        \item $a_k'\leq-2\implies\xi_{k+1}'\in(0,\tfrac12)$
        \item $a_k=1\implies\xi_{k+1}'\in(-2,-1)$
    \end{enumerate}
  \item if $\xi_k'\in(-2,-1)$, then
     \begin{enumerate}
	\item $a_k'>0\implies\xi_{k+1}'\in(-1,0)$
	\item $a_k'\leq-4\implies\xi_{k+1}'\in(0,\tfrac12)$
     \end{enumerate}
\end{enumerate}

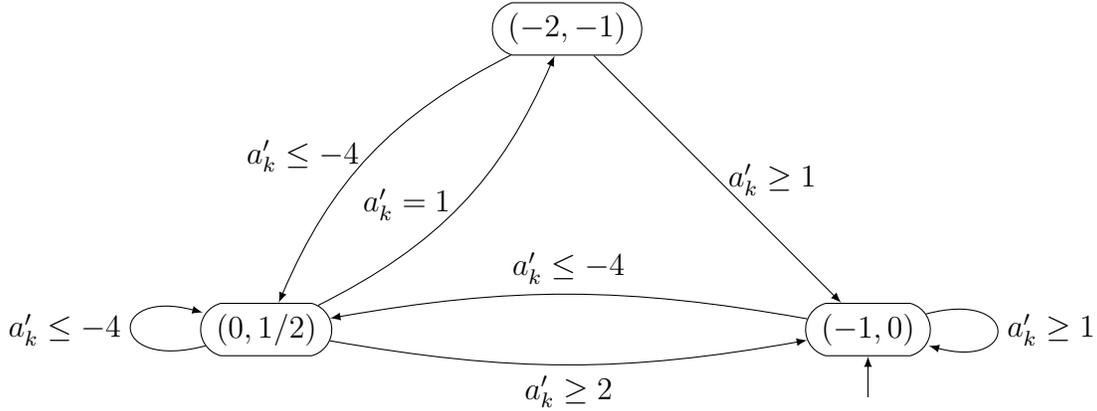
\begin{figure}[ht]
\centering
	\begin{tikzpicture}[auto, initial text=, >=latex]
	\node[draw, rounded rectangle, rounded rectangle arc length=150] (v0) at (-4, 0) {$(0,1/2)$};
	\node[draw, rounded rectangle, rounded rectangle arc length=150] (v1) at (0, 4) {$(-2,-1)$};
	\node[draw, rounded rectangle, rounded rectangle arc length=150    , initial,initial where=below] (v2) at (4, 0) {$(-1,0)$};
    \path[->] (v0) edge[loop left] node {$a'_k\leq -4$} ();
	\path[->] (v0) edge[bend right=20] node[left] {$a'_k=1$} (v1);
	\path[->] (v0) edge[bend right=10] node[below] {$a'_k\geq 2$} (v2);
	\path[->] (v1) edge[bend right=20] node[left] {$a'_k\leq -4$} (v0);
	\path[->] (v1) edge node[right] {$a'_k\geq 1$} (v2);
	\path[->] (v2) edge[bend right=10] node[above] {$a'_k\leq -4$} (v0);
   	\path[->] (v2) edge[loop right] node {$a'_k\geq 1$} ();
	\end{tikzpicture}
	\caption{Labeled graph representing the dynamics of the complete quotients in the $\beta$-continued fraction of $\xi$.}
	\label{fig:automat}
\end{figure}
\end{proof}

Theorem~\ref{t:negatconjugnotCFF} covers the cases when $\beta$ is a root of~\eqref{eq:betanegatconjug} with $b\geq 4$, but for $b=1,2,3$ we have to disprove property \CFF{} by other means. Mercat's conjecture implies that if $\beta>3$ \CFF{} cannot hold. This leaves only 30 values of $\beta$: the positive roots of polynomials
$$
X^2+bX-c,\quad b=1,2,3,\quad b+2\leq c < 3(b+3),\ c\neq 2(b+2).
$$
For 10 of these values we have found elements of the corresponding quadratic field $\Q(\beta)$ with periodic $\beta$-continued fraction expansion, thus refuting $\CFF$ directly.
However, there remain 20 cases for which the question of property $\CFF$ remains open.
The results supported by of our computational experiments are summarized in Table~\ref{t2}.
\begin{table}[h!]
$$
\begin{array}{|c|c|c|c|}\hline
\multicolumn{2}{|c|}{} & \CFF & \text{Reason} \\\hline
b=1 & c\geq 12  & \text{No} & \text{Implied by Mercat's conjecture} \\
b=1 & 5\leq c\leq 11,\,c\neq6  &  \text{No}   & \text{We found a periodic expansion}  \\
b=1 & c=3,4  &  \text{Open}  & \text{}   \\
b=2 & c\geq 16 & \text{No}  & \text{Implied by Mercat's conjecture}\\
b=2 & 12\leq c\leq 14  &  \text{No}  & \text{We found a periodic expansion}\\
b=2 & 4\leq c\leq 11,\,c\neq 8  &  \text{Open}   & \text{} \\
b=3 & c\geq 18 & \text{No} & \text{Implied by Mercat's conjecture}\\
b=3 & c=17  &  \text{No}   & \text{We found a periodic expansion}\\
b=3 & 5\leq c\leq 16,\,c\neq 10  &  \text{Open}   & \text{} \\
b\geq4 & c\geq b+2 & \text{No} & \text{Theorem~\ref{t:negatconjugnotCFF}} \\\hline
\end{array}
$$
  \caption{Expected behaviour of the $\beta$-continued fractions for non-Perron quadratic numbers with negative conjugate in terms
  of the parameters of its minimal polynomial $x^2+bx-c$.}\label{t2}
\end{table}

Let us comment on what our computational experiments suggest on eventually periodic $\beta$-continued fraction expansions of elements of $\Q(\beta)$. Given $b\geq 1$, for sufficiently large $c$ we expect that property $\CFP$ holds. For intermediate values of $c$, we can find in $\Q(\beta)$ both elements with eventually periodic and aperiodic $\beta$-continued fraction expansion. For the smallest admissible values of $c\geq b+2$, it is likely that all elements of $\Q(\beta)$ have either finite or aperiodic $\beta$-continued fraction expansion.
The exact behaviour is yet to be investigated.


\appendix
\section{}
Let us present here an alternative proof of the finiteness result in Theorem~\ref{t:superteorem} which provides a better bound for the \emph{number} of irrational partial quotients.

In the entire appendix, let $\xi=[a_0,a_1,\dotsc]$ be an infinite continued fraction with the value in a quadratic field $K$, $a_n\in\O_K$ and $a_n\geq 1$ for all $n$. Let $d$ be the leading coefficient of the minimal polynomial of $\xi$ over $\Z$ (which we take to be positive). We will always denote by $x'$ the Galois conjugate of $x\in K$.

\begin{thm}\label{thm:appendix}
 Assume that $\abs{a_n'}<a_n$ for all $n$ such that $a_n\not\in\Z$. Then there are at most
 \[
  36H(\xi)^2\log H(\xi)+9\log(3) H(\xi)^2
 \]
irrational partial quotients, and therefore the given continued fraction is eventually periodic.
\end{thm}

\medskip

Before plunging into the proof, let us briefly recall why all rational numbers have a finite regular continued fraction expansion.

On one hand, continued fractions give very good rational approximations, and if $p/q$ is a convergent of a rational number $a/b$ then the majoration $\big|\frac{a}{b}-\frac{p}{q}\big|\leq \frac{1}{q^2}$ holds.
On the other hand, rational numbers are badly approximated by other rational numbers and if $\frac{a}{b}\neq\frac{p}{q}$ then a minoration $\big|\frac{a}{b}-\frac{p}{q}\big|\geq \frac{\abs{aq-bp}}{bq}\geq\frac{1}{bq}$ holds. The two inequalities taken together imply that for a fixed $a/b$ only finitely many distinct convergents $p/q$ may exist.

The same majoration also holds for any convergent continued fractions, as seen with \eqref{eq:odhad}, so in order to repeat the classical proof we need an argument of diophantine approximation to supply a minoration; this is given by the following proposition.
\begin{prop}\label{prop:DiophApprox}
 For all $n\geq 0$ the following inequality holds:
 \[
  \abs{\frac{q_n'}{q_n}}\abs{\xi'-\frac{p_n'}{q_n'}}> \frac{a_n}{d^2}.
 \]
\end{prop}
\begin{proof}
 The quantity $q_n\xi-p_n$ is not equal to zero because  otherwise the expansion of $\xi$ would be finite.
 As $d$ is the leading coefficient of the minimal polynomial of $\xi$ over $\Z$, we see that $d \xi\in\O_K$.
 Then the norm of $d(q_n\xi-p_n)$ must be at least one in absolute value, because $p_n$ and $q_n$ are algebraic integers.
 So we have that
 \begin{equation*}
  1\leq \abs{N(d(q_n\xi-p_n))}=d^2\abs{q_n\xi-p_n}\cdot\abs{q_n'\xi'-p_n'}<\frac{d^2}{a_n} \abs{\frac{q_n'}{q_n}}\cdot \abs{\xi'-\frac{p_n'}{q_n'}},
 \end{equation*}
where the last inequality follows from \eqref{eq:odhad}.
\end{proof}
As the euclidean absolute value over $\Q$ splits into two archimedean absolute values over the quadratic number field $K$, our minoration involves both the convergents $p/q$ and their conjugates $p'/q'$. The conjugates $p'$ and $q'$ obey the same recurrence relations as $p$ and $q$, with the partial quotients replaced by their conjugates. These conjugates however need not be bounded away from 0, or even be positive, and we cannot guarantee in general that the ratio $p'/q'$ will converge to some value. In order to proceed with the proof we need a way of controlling their growth, which we will do through an accurate analysis of the recurrences~\eqref{eq:pnqn}.

\smallskip

We further state here two easily-checked remarks, which we will use in the proof below:
\begin{remark}\label{cl:1}
  Let $0<A<B$ be real numbers. The function $f(x)=\frac{A+x}{B+x}$ is strictly increasing on $(0,+\infty)$.
\end{remark}

\begin{remark}\label{cl:2}
  Let $(\alpha_n)_{n\geq 0}$ satisfy a linear recurrence of order two. Then for every $n\geq 0$, $\alpha_n=f_n\alpha_1+g_n\alpha_0$, where $(f_n)_{n\geq 0}$, $(g_n)_{n\geq 0}$ satisfy the same linear recurrence with initial conditions $f_0=g_1=0$, $f_1=g_0=1$.
\end{remark}

\begin{lem}\label{lem:posloupnosti}
  Let $(\alpha_n)_{n\geq 0}$, $(\beta_n)_{n\geq 0}$ be two real sequences satisfying the same linear recurrence relation of order two
  \begin{equation}\label{eq:spolrek}
  x_n = r_n x_{n-1} + x_{n-2},\qquad n\geq 2,
  \end{equation}
with $r_i\geq 1$, for $i\geq 1$, with initial conditions $\alpha_{1}>\beta_{1}>0$, $\alpha_{0}=\beta_0>0$.
Then for all $n\geq 3$ we have
$$
\frac{\beta_n}{\alpha_n} < \frac{\beta_{1}+\beta_{0}}{\alpha_{1}+\alpha_{0}}\,.
$$
\end{lem}

\pfz
Let $(f_n)_{n\geq 0}$, $(g_n)_{n\geq 0}$ be the sequences from Remark~\ref{cl:2} such that $\alpha_n=f_n\alpha_1+g_n\alpha_0$, $\beta_n=f_n\beta_1+g_n\beta_0=f_n\beta_1+g_n\alpha_0$.
Since $(f_n)_{n\geq 0}$, $(g_n)_{n\geq 0}$ satisfy~\eqref{eq:spolrek}, they are both positive sequences for $n\geq 2$. By induction,  with the use of $r_i\geq 1$, one can show that  $g_n<f_n$ for $n\geq 3$.
Using Remark~\ref{cl:1}, we then have
$$
\frac{\beta_n}{\alpha_n}=\frac{f_n\beta_1+g_n\beta_0}{f_n\alpha_1+g_n\alpha_0} < \frac{f_n\beta_1+f_n\beta_0}{f_n\alpha_1+f_n\alpha_0} = \frac{\beta_1+\beta_0}{\alpha_1+\alpha_0}\,,
$$
where we have used that $\alpha_0=\beta_0$.
\pfk

We now prove a proposition which allows to compare recurrent sequences of the shape \eqref{eq:pnqn}. We will apply it to the sequences $(p_n)_n,(q_n)_n$ and their conjugates.

\begin{prop}\label{t:lim}
  Let $(a_n)_{n\geq 0}$ be a sequence of positive reals and $(b_n)_{n\geq 0}$ be a sequence of complex numbers such that $a_n\geq 1$ and $\abs{b_n}\leq a_n$ for every $n\geq 0$.

  Let $(s_n)_{n}$, $(t_n)_{n}$ satisfy for every $n\geq 0$
$$
\begin{aligned}
s_n & = a_n s_{n-1} + s_{n-2},&\quad& s_{-1}=t_{-1}>0 & \\
t_n & = b_n t_{n-1} + t_{n-2},&\quad & s_{-2} = t_{-2} = 0. &
\end{aligned}
$$
Let $0<C<1$ and $r_n=\#\{0\leq k \leq n \mid \abs{b_k}<C a_k\}$. Then for all $n\geq 4$ we have
\begin{equation*}
 \frac{\abs{t_n}}{s_n}\leq \left(\frac{C+2}{3}\right)^{r_{n-3}}.
\end{equation*}

In particular, if $\frac{|b_k|}{a_k}<C$ for infinitely many $k\in\N$, then $$\lim_{n\to+\infty}\frac{t_n}{s_n}=0.$$
\end{prop}

\begin{proof}
Let us define for each $k\in\N$ an auxiliary sequence $(\gamma^{(k)})_{n}$ as follows,
$$
\begin{aligned}
\gamma^{(k)}_{-2}&=0, & \gamma^{(k)}_{-1}=s_{-1},\\
\gamma^{(k)}_n &=|b_n| \gamma^{(k)}_{n-1} + \gamma^{(k)}_{n-2}, &\text{ for $0\leq n< k$},\\
\gamma^{(k)}_n &=a_n \gamma^{(k)}_{n-1} + \gamma^{(k)}_{n-2},  &\text{ for $n\geq k$}.
\end{aligned}
$$
Clearly each sequence $(\gamma^{(k)})_{n\in\N}$ is strictly increasing for $n\geq 0$ and  we also have
\begin{equation}\label{eq:spoustuposloupnosti}
|t_n|=\gamma^{(n+1)}_{n}\leq \cdots \leq\gamma^{(k+1)}_n \leq \gamma^{(k)}_n \leq \cdots \leq\gamma^{(0)}_n =s_n\quad\text{ for every $n\in\N$}.
\end{equation}

Moreover, we have $\gamma^{(k+1)}_n = \gamma^{(k)}_n$ for $n=0,\dots,k-1$, and if $|b_k|<a_k$, then
$\gamma^{(k+1)}_k < \gamma^{(k)}_k$.

Now consider a fixed index $k\geq 2$ such that $\frac{|b_k|}{a_k}<C$.
We will show that there exists a constant $K<1$ such that
\begin{equation}\label{eq:nekonecnemensiK}
\frac{\gamma_n^{(k+1)}}{\gamma_n^{(k)}}<K \qquad \text{for every } n\geq k+2.
\end{equation}
We will apply Lemma~\ref{lem:posloupnosti} with $\alpha_n=\gamma_{k+n-1}^{(k)}$, $\beta_n=\gamma_{k+n-1}^{(k+1)}$ for $n\geq 0$.
These sequences satisfy the same recurrence with $r_n=a_{k+n-1}$, and the assumptions $\alpha_{1}>\beta_{1}>0$, $\alpha_{0}=\beta_0>0$
on initial conditions are satisfied.

For all $n\geq k+2$ we obtain
$$
\begin{aligned}
\frac{\gamma_n^{(k+1)}}{\gamma_n^{(k)}}&=\frac{\beta_{n-k+1}}{\alpha_{n-k+1}}< \frac{\beta_1+\beta_0}{\alpha_1+\alpha_0}=
\frac{\gamma_{k}^{(k+1)}+\gamma_{k-1}^{(k+1)}}{\gamma_{k}^{(k)}+\gamma_{k-1}^{(k)}} =
\frac{|b_k|\gamma_{k-1}^{(k+1)}+\gamma_{k-2}^{(k+1)}+ \gamma_{k-1}^{(k+1)}}{a_k\gamma_{k-1}^{(k)}+\gamma_{k-2}^{(k)}+\gamma_{k-1}^{(k)}}=\\
&=\frac{(|b_k|+1)\gamma_{k-1}^{(k+1)}+\gamma_{k-2}^{(k+1)}}{(a_k+1)\gamma_{k-1}^{(k)}+\gamma_{k-2}^{(k)}}=
\frac{|b_k|+1+\frac{\gamma_{k-2}^{(k+1)}}{\gamma_{k-1}^{(k+1)}}}{a_k+1+\frac{\gamma_{k-2}^{(k)}}{\gamma_{k-1}^{(k)}}},
\end{aligned}
$$
where we have used that $\gamma^{(k+1)}_{k-1} = \gamma^{(k)}_{k-1}$. Since the sequences $(\gamma_{n}^{(j)})$ are strictly increasing, we derive by Remark~\ref{cl:1} that
$$
\frac{\gamma_n^{(k+1)}}{\gamma_n^{(k)}}<\frac{|b_k|+2}{a_k+2}=\frac{\frac{|b_k|}{a_k}+\frac{2}{a_k}}{1+\frac{2}{a_k}}\leq \frac{\frac{|b_k|}{a_k}+2}{3}<\frac{C+2}{3}=:K\,,
$$
where we use again Remark~\ref{cl:1} with the fact that $a_k\geq 1$ and $k$ is such that $\frac{|b_k|}{a_k}<C$. Since $C<1$, also $K<1$, which shows~\eqref{eq:nekonecnemensiK} is true.

\smallskip
By~\eqref{eq:spoustuposloupnosti} and~\eqref{eq:nekonecnemensiK}, for any $k\geq 2$ and $n\geq k+2$, we have
$$
\frac{|t_n|}{s_n}\leq \frac{\gamma_n^{(k)}}{\gamma_n^{(0)}}=\prod_{j=0}^{k-1} \frac{\gamma_n^{(j+1)}}{\gamma_n^{(j)}} \leq
\prod_{j=0}^{k-1} K_j \,,
$$
where $K_j=K$ for $j$ such that $\frac{|b_j|}{a_j}<C$ and $K_j=1$ otherwise. Thus for every $n\geq 4$,
$$
0\leq  \frac{|t_n|}{s_n}\leq \prod_{j=0}^{n-3} K_j=K^{r_{n-3}}.
$$

If $r_n$ tends to infinity with $n$, we obtain $\lim_{n\to\infty} \frac{|t_n|}{s_n}=0$.
\end{proof}

We are now ready for the final part of the argument.

\begin{proof}[Proof of Theorem~\ref{thm:appendix}]
We first show that if $|a'_n|<a_n$ for every irrational partial quotient $a_n$, then there exists a constant $0<C<1$ such that for all $n\geq 0$ with $a_n\notin\Z$ we have $\frac{\abs{a_n'}}{a_n}\leq C$.

Since $|a'_n|\leq a_n$ for all $n\geq 0$, by Proposition~\ref{l:completequobound} the partial quotients satisfy $a_n\leq 3H(\xi)^2$. Since $a_n,a_n'$ are algebraic integers, the quantities $a_n+a_n'$ and $a_n-a_n'$ are in $\Z\cup \sqrt{D}\Z$, so that $a_n-\abs{a_n'}\geq 1$ whenever it is not equal to 0. Dividing by $a_n$ we see that
\[
 \frac{\abs{a_n'}}{a_n}\leq 1-\frac{1}{a_n}\leq 1-\frac{1}{3H(\xi)^2}=:C.
\]

We can now apply Proposition~\ref{t:lim} with $(a_n')_n$ as the sequence $(b_n)_n$ and $(p_n)_n, (p_n')_n$  as the sequences $(s_n)_n, (t_n)_n$. Analogously but shifting all indices by one, so that the initial conditions are verified, we can do it for $(q_{n+1})_n, (q_{n+1}')_n$.
Therefore for all $n\geq 5$  we have that
\begin{equation*}
 \frac{\abs{p_n'}}{p_n},\frac{\abs{q_n'}}{q_n}\leq \left(\frac{C+2}{3}\right)^{r_{n-4}},
\end{equation*}
where  $r_n=\#\{0\leq k \leq n \mid a_k\not\in\Z\}$.

This implies that
\[
 \abs{\frac{q_n'}{q_n}}\cdot\abs{\xi'-\frac{p_n'}{q_n'}}\leq \abs{\frac{q_n'}{q_n}}\abs{\xi'}+\abs{\frac{p_n'}{p_n}}\abs{\frac{p_n}{q_n}}\leq \left(\frac{C+2}{3}\right)^{r_{n-4}}(\abs{\xi'}+\abs{\xi}+1).
\]

On the other hand, by Proposition~\ref{prop:DiophApprox}
\[\abs{\frac{q_n'}{q_n}}\cdot\abs{\xi'-\frac{p_n'}{q_n'}}> \frac{1}{d^2},\]
so combining them we obtain
\[
 \left(\frac{C+2}{3}\right)^{-r_{n-4}}<d^{2} (\abs{\xi'}+\abs{\xi}+1)\leq 3 H(\xi)^4,
\]
where the second inequality comes from Remark~\ref{rem:Mahlermeasure} and
$1+a+b\leq 3 \sup(1,a)\sup(1,b)$ for all $a,b\geq 0$.

Taking logarithms we obtain
\[
 -r_{n-4}\log\left(1-\frac{1}{9H(\xi)^2}\right)<4\log H(\xi) +\log 3,
\]
and remembering that $\log(1+x)\leq x$ for all $x>-1$ we obtain
\[
 r_{n-4}<36H(\xi)^2\log H(\xi)+9\log(3) H(\xi)^2.
\]
Letting $n$ go to infinity, this shows that the number of irrational partial quotients is finite and does not exceed the stated bound.

As soon as the tail of the expansion contains only positive integers, the classical theory of simple continued fractions and the fact that $\xi$ lies in a quadratic field imply that the expansion is eventually periodic.
\end{proof}
Notice that, while this theorem gives a bound for the \emph{number} of irrational partial quotients, it does not tell us \emph{where} they are.
In comparison, Theorem~\ref{t:superteorem} shows that the irrational partial quotients can be effectively found.
Nevertheless we wanted to include this appendix because the proof given here is self-contained, rather elementary, and gives a bound which is substantially lower than those alluded to in Remark~\ref{rem:number.bounded.height}, which rely on much more sophisticated techniques.

\section*{Acknowledgements}
This work was supported by the project CZ.02.1.01/0.0/0.0/16\_019/0000778 of the Czech Technical University in Prague
and projects PRIMUS/20/SCI/002 and UNCE/SCI/022 from Charles University.
We thank the Centro di Ricerca Ma\-te\-ma\-ti\-ca Ennio De Giorgi for support.

The third author is a member of the GNSAGA research group of INdAM.

\bibliography{reference}

\begin{bibdiv}
\begin{biblist}

\bib{AkiyamaPisotAndGreedy}{incollection}{
      author={Akiyama, Shigeki},
       title={Pisot numbers and greedy algorithm},
        date={1998},
   booktitle={Number theory ({E}ger, 1996)},
   publisher={de Gruyter, Berlin},
       pages={9\ndash 21},
      review={\MR{1628829}},
}

\bib{elkies}{article}{
      author={Brock, Bradley~W.},
      author={Elkies, Noam~D.},
      author={Jordan, Bruce~W.},
       title={Periodic continued fractions over {$S$}-integers in number fields
  and {S}kolem's {$p$}-adic method},
        date={2019},
         url={https://arxiv.org/abs/1909.11214},
}

\bib{bernat}{article}{
      author={Bernat, Julien},
       title={Continued fractions and numeration in the {F}ibonacci base},
        date={2006},
        ISSN={0012-365X},
     journal={Discrete Math.},
      volume={306},
      number={22},
       pages={2828\ndash 2850},
         url={https://doi.org/10.1016/j.disc.2006.05.020},
      review={\MR{2261781}},
}

\bib{BuFrGaKr}{article}{
      author={Burd\'{i}k, \v{C}estm{\'\i}r},
      author={Frougny, Christiane},
      author={Gazeau, Jean-Pierre},
      author={Krejcar, Rudolf},
       title={Beta-integers as natural counting systems for quasicrystals},
        date={1998},
        ISSN={0305-4470},
     journal={J. Phys. A},
      volume={31},
      number={30},
       pages={6449\ndash 6472},
         url={https://doi.org/10.1088/0305-4470/31/30/011},
      review={\MR{1644115}},
}

\bib{BG}{book}{
      author={Bombieri, Enrico},
      author={Gubler, Walter},
       title={Heights in {D}iophantine geometry},
      series={New Mathematical Monographs},
   publisher={Cambridge University Press, Cambridge},
        date={2006},
      volume={4},
        ISBN={978-0-521-84615-8; 0-521-84615-3},
         url={https://doi.org/10.1017/CBO9780511542879},
      review={\MR{2216774}},
}

\bib{feng}{article}{
      author={Feng, Jing},
      author={Ma, Chao},
      author={Wang, Shuailing},
       title={Metric theorems for continued $\beta$-fractions},
        date={2019},
     journal={Monatsh. Math.},
      volume={190},
       pages={281\ndash 299},
         url={https://link.springer.com/article/10.1007/s00605-019-01305-6},
}

\bib{FroSo}{article}{
      author={Frougny, Christiane},
      author={Solomyak, Boris},
       title={Finite beta-expansions},
        date={1992},
        ISSN={0143-3857},
     journal={Ergodic Theory Dynam. Systems},
      volume={12},
      number={4},
       pages={713\ndash 723},
         url={https://doi.org/10.1017/S0143385700007057},
      review={\MR{1200339}},
}

\bib{slicing}{article}{
      author={Grizzard, Robert},
      author={Gunther, Joseph},
       title={Slicing the stars: counting algebraic numbers, integers, and
  units by degree and height},
        date={2017},
        ISSN={1937-0652},
     journal={Algebra Number Theory},
      volume={11},
      number={6},
       pages={1385\ndash 1436},
         url={https://doi.org/10.2140/ant.2017.11.1385},
      review={\MR{3687101}},
}

\bib{ConcreteMaths}{book}{
      author={Graham, Ronald~L.},
      author={Knuth, Donald~E.},
      author={Patashnik, Oren},
       title={Concrete mathematics: A foundation for computer science},
     edition={2},
   publisher={Addison-Wesley Longman Publishing Co., Inc.},
     address={USA},
        date={1994},
        ISBN={0201558025},
}

\bib{HbKa}{article}{
      author={Hbaib, M.},
      author={Kammoun, R.},
       title={Continued {$\beta$}-fractions with formal power series over
  finite fields},
        date={2016},
        ISSN={1382-4090},
     journal={Ramanujan J.},
      volume={39},
      number={1},
       pages={95\ndash 105},
         url={https://doi.org/10.1007/s11139-015-9725-5},
      review={\MR{3439828}},
}

\bib{HarWri}{book}{
      author={Hardy, G.~H.},
      author={Wright, E.~M.},
       title={An introduction to the theory of numbers},
     edition={Sixth},
   publisher={Oxford University Press, Oxford},
        date={2008},
        ISBN={978-0-19-921986-5},
        note={Revised by D. R. Heath-Brown and J. H. Silverman, With a foreword
  by Andrew Wiles},
}

\bib{Kolar}{book}{
      author={Kol{\' a}{\v r}, Jakub},
       title={Continued fractions based on beta-integers (in czech)},
   publisher={Bachelor's project -- Czech Technical University in Prague},
        date={2011},
}

\bib{McMullen}{article}{
      author={McMullen, Curtis~T.},
       title={Uniformly {D}iophantine numbers in a fixed real quadratic field},
        date={2009},
        ISSN={0010-437X},
     journal={Compos. Math.},
      volume={145},
      number={4},
       pages={827\ndash 844},
         url={https://doi.org/10.1112/S0010437X09004102},
      review={\MR{2521246}},
}

\bib{mercat}{article}{
      author={Mercat, Paul},
       title={Construction de fractions continues p\'{e}riodiques
  uniform\'{e}ment born\'{e}es},
        date={2013},
        ISSN={1246-7405},
     journal={J. Th\'{e}or. Nombres Bordeaux},
      volume={25},
      number={1},
       pages={111\ndash 146},
}

\bib{Parry}{article}{
      author={Parry, William},
       title={On the {$\beta $}-expansions of real numbers},
        date={1960},
        ISSN={0001-5954},
     journal={Acta Math. Acad. Sci. Hungar.},
      volume={11},
       pages={401\ndash 416},
         url={https://doi.org/10.1007/BF02020954},
}

\bib{Renyi}{article}{
      author={R\'{e}nyi, Alfred},
       title={Representations for real numbers and their ergodic properties},
        date={1957},
        ISSN={0001-5954},
     journal={Acta Math. Acad. Sci. Hungar},
      volume={8},
       pages={477\ndash 493},
         url={https://doi.org/10.1007/BF02020331},
}

\bib{rosen}{article}{
      author={Rosen, David},
       title={A class of continued fractions associated with certain properly
  discontinuous groups},
        date={1954},
        ISSN={0012-7094},
     journal={Duke Math. J.},
      volume={21},
       pages={549\ndash 563},
         url={http://projecteuclid.org/euclid.dmj/1077465884},
      review={\MR{65632}},
}

\bib{rosen-amm}{article}{
      author={Rosen, David},
       title={Research {P}roblems: {C}ontinued {F}ractions in {A}lgebraic
  {N}umber {F}ields},
        date={1977},
        ISSN={0002-9890},
     journal={Amer. Math. Monthly},
      volume={84},
      number={1},
       pages={37\ndash 39},
         url={https://doi.org/10.2307/2318305},
}

\bib{schanuel}{article}{
      author={Schanuel, Stephen~Hoel},
       title={Heights in number fields},
        date={1979},
        ISSN={0037-9484},
     journal={Bull. Soc. Math. France},
      volume={107},
      number={4},
       pages={433\ndash 449},
         url={http://www.numdam.org/item?id=BSMF_1979__107__433_0},
      review={\MR{557080}},
}

\end{biblist}
\end{bibdiv}
\bibliographystyle{amsalpha}
\end{document}